\documentclass[12pt]{amsart}
\usepackage{graphicx}
\usepackage{color}
\allowdisplaybreaks
\newtheorem{thm}{Theorem}[section]
\newtheorem{corollary}[thm]{Corollary}
\newtheorem{lemma}[thm]{Lemma}
\newtheorem{prop}[thm]{Proposition}
\theoremstyle{definition}
\newtheorem{definition}[thm]{Definition}
\newtheorem{example}[thm]{Example}
\theoremstyle{remark}

\numberwithin{equation}{section}

\newcommand{\abs}[1]{\left\vert#1\right\vert}

\newcommand{\N}{\mathbb{{N}}}

\begin{document}

\title[Vector-valued fractional differential-difference equations]
{On well-posedness of vector-valued fractional differential-difference equations}


\author[Abadias]{Luciano Abadias}
\address{Departamento de Matem\'aticas, Instituto Universitario de Matem\'aticas y Aplicaciones, Universidad de Zaragoza, 50009 Zaragoza, Spain.}
\email{labadias@unizar.es}

\author[Lizama]{Carlos Lizama}
\address{Universidad de Santiago de Chile, Facultad de Ciencias, Departamento de Matem\'atica y Ciencia de la Computaci\'on, Casilla
 307, Correo 2, Santiago, CHILE} \email{carlos.lizama@usach.cl}

\author[Miana]{Pedro J. Miana}
\address{Departamento de Matem\'aticas, Instituto Universitario de Matem\'aticas y Aplicaciones, Universidad de Zaragoza, 50009 Zaragoza, Spain.}
\email{pjmiana@unizar.es}

\author[Velasco]{M. Pilar Velasco}
\address{\'Area de Matem\'aticas, Estad\'istica e Investigaci\'on Operativa, Centro Universitario de la Defensa, Instituto Universitario de Matem\'aticas y Aplicaciones, 50090 Zaragoza, Spain.}
\email{velascom@unizar.es}
\thanks{M.P. Velasco has been partially supported by Universidad de Santiago de Chile, Proyecto
POSTDOC-DICYT, Codigo  041533LY, Vicerrector\'ia de Investigaci\'on,
Desarrollo e Innovaci\'on.}

\thanks{C. Lizama has been partially supported by DICYT, Universidad de Santiago de Chile; Project
CONICYT-PIA ACT1112 Stochastic Analysis Research Network; FONDECYT 1140258 and Ministerio de Educaci\'{o}n CEI Iberus (Spain). L. Abadias and P. J. Miana  have been partially supported by Project MTM2013-42105-P, DGI-FEDER, of the MCYTS; Project E-64, D.G. Arag\'on, and  UZCUD2014-CIE-09, Universidad de Zaragoza, Spain.}

\subjclass[2010]{35R11; 65Q10; 47D06}

\keywords{Difference equations; fractional and nonlinear PDE; Poisson distribution; weighted Lebesgue space}


\begin{abstract}
We develop an operator-theoretical method  for the analysis on well posedness of partial differential equations that can be modeled in the form
\begin{equation*}
(*) \left\{%
\begin{array}{rll}
   \Delta^{\alpha} u(n) &=  Au(n+2) + f(n,u(n)),  \quad n \in \mathbb{N}_0, \,\, 1< \alpha \leq 2;  \\
    u(0) &= u_0;\\
    u(1) &= u_1,
\end{array}%
\right.
\end{equation*}
where $A$ is an closed linear operator defined on a Banach space $X$. Our ideas are inspired on the Poisson distribution as a tool to sampling fractional  differential operators into fractional differences. Using our abstract approach, we are able to show existence and uniqueness of solutions for the problem (*) on a distinguished class of weighted Lebesgue spaces of sequences, under mild  conditions on strongly continuous sequences of bounded operators generated by $A,$ and natural restrictions on the nonlinearity $f$. Finally we present some original examples to illustrate our results. \end{abstract}

\maketitle

\section{Introduction}


In an interesting  paper  published in 1943, H. Bateman \cite{Ba43}  studied the differential-difference equation
\begin{equation}
\begin{array}{lcl}\label{eq:bateman}
M  u_{xx}(n,x)+ 2K u_x(n,x) + S u(n,x)   & = & (na+a+b)[u(n+1,x)-u(n,x)] \\ & - & (na+c)[u(n,x)-u(n-1,x)],
\end{array}
\end{equation}
where $M,K,S, a,b,c $ are positive constants. This general formulation includes previous work of B. Taylor, J. Bernoulli and D. Bernoulli. In particular, it includes the temporal discretization of the  diffusion equation
$$
u(n+1)-u(n,x)= u_{xx}(n,x),
$$
and the  wave equation
$$
u(n+1,x)-2u(n,x)+u(n-1,x) = u_{xx}(n,x).
$$
Defining $v(n)(x):= u(n,x)$ and $A=\displaystyle\frac{\partial^2}{\partial x^2}$ the diffusion and wave equations can be rewritten as
$$
\Delta v(n)= Av(n), \qquad 
\hbox{and}\qquad 
\Delta^2 v(n)= Av(n+1),\qquad n\in \N,
$$
respectively, where $A$ is the generator of a $C_0$-semigroup of operators given by
$$
T(t)f(x)= \frac{1}{\sqrt{4\pi t}} \int_{-\infty}^{\infty} e^{\frac{-(x-s)^2}{4t}}f(s)ds.
$$
The difference-differential equation
\begin{equation}\label{eq-a}
 \quad \lambda p(n+1,x) - \lambda p(n,x)= -p_x(n+1,x)
\end{equation}
is even covered by \eqref{eq:bateman}. With initial value $p(0,x)= e^{-\lambda x}$ and border condition $p(n,0)=0, \,\, n=1,2,...$
it defines the  probability distribution function of the Poisson process given by
$$
p(n,x) = \frac{\lambda^n x^n}{n!} e^{-\lambda x}, \quad n \in \mathbb{N}.
$$
Note that equation $\eqref{eq-a}$ is likewise used in the so called car-following problems, see for example \cite{Ch-He-Mo}.

Defining $u(n)(x):= \lambda p(n,x)$ and $A=-\frac{\partial}{\partial x}$  equation $\eqref{eq-a}$ can be rewritten in abstract form as
$$
\Delta u(n) = Au(n+1), \quad n \in \mathbb{N}_0.
$$
Note that $A$ is the generator of the translation semigroup $T(t)f(x)=f(x -t).$

The difference-differential equation
\begin{equation}\label{eq-b}
 \quad T(n+2,x)=2xT(n+1)-T(n,x), \quad n \in \mathbb{N}_0,
\end{equation}
with appropriate initial values define  Chebyshev polynomials of the first and second kind. Indeed, take $T(0,x)=1$ and $T(1,x)=x$ in the first case and
$T(0,x)=1$and $T(1,x)=2x$ in the second case.

Defining $u(n)(x):= T(n,x),$ equation \eqref{eq-b} can be rewritten in abstract form as
$$
\Delta^2 u(n) = Au(n+1), \quad n \in \mathbb{N}_0,
$$
where $Af(x)=2(x-1)f(x).$  We observe that  $A$ is the generator of the multiplication semigroup: $T(t)f(x)= e^{2(x-1)}f(x).$

These abstract models, with unbounded operators $A$ defined on Banach spaces, are closely connected with numerical methods for partial differential equations, integro-differential equations \cite{Cu-Lu-Pa06, Sanz88} and evolution equations \cite{Lu-Sl-Th96}, \cite{Mi93}. See also the recent monograph \cite{ACL14}.  Recently, it has been shown that the extension of some of this models to  fractional difference equations is  a promising tool for several biological and physical applications where
a memory effect appears \cite{AS2010, Ba-Di-Sc12}.

In spite of the significant increase of research in this area,
there are still many significative questions regarding fractional difference equations. In particular, the study  of  fractional difference equations with closed linear operators and their qualitative behaviour remains an open problem.

In this paper, we propose a novel method to deal with this classes of abstract fractional difference problems. This method is inspired on sampling by means of the Poisson distribution.  We will use it to develop a theory on  well-posedness for the abstract fractional difference problem
\begin{equation}\label{eq1}
\left\{%
\begin{array}{rll}
    \Delta^{\alpha} u(n) &=  Au(n+2)+f(n,u(n)),  \quad n \in \mathbb{N}_0, \quad 1 < \alpha\le2;  \\
    u(0) &= u_0 \in X; \\ u(1) & = u_1 \in X,
\end{array}%
\right.
\end{equation}
where $A:D(A)\subset X \to X$ is a closed linear operator defined on a Banach space $X$ and $f$ a suitable function.


We notice that first studies on the model \eqref{eq1} when $A$ is a complex or real valued matrix, have only recently appeared \cite{AE2011, Da-Ba-Ka14}. However,  the study of this equation when $A$ is a closed linear operator, not necessarily bounded, has not been considered in the literature.

The approach followed here is purely operator-theoretic and has as main ingredient the use of the Poisson distribution
$$
p_n(t)= e^{-t} \frac{t^n}{n!}, \quad n \in \mathbb{N}_0, \quad t \geq 0.
$$
Our method relies in to take advantage of the properties of this distribution  when it is applied to discrete phenomena. More precisely, given a continuous evolution $(u(t))_{t \in [0, \infty)}$ we can discretize it by means of that we will call the {\it Poisson transformation}
\begin{equation}\label{eq1.2}
(\mathcal{P}u)(n):= \int_0^{\infty} p_n(t)u(t)dt, \quad n \in \mathbb{N}_0.
\end{equation}
In this paper, we will show that when this procedure is applied to continuous fractional processes, these transformations are well behaved and fit perfectly in the discrete fractional concepts. A remarkable feature of this work will be to show that by this method of sampling  we recover the concept of fractional nabla sum and difference operator in \cite{AE2009b}, which has been used recently and independently of the method used here by other authors in order to obtain  qualitative properties of fractional difference equations, notably concerning stability properties \cite{CeKiNe12, CeKiNe13}.

The outline of this paper is as follows: In Section 2, we give some  background on the definitions to be used. The remarkable fact here is that we highlight a particular choice of the definition introduced  in \cite{AtEl11} for the nabla operator; see the recent paper \cite{Li15}. This choice, that has been implicitly used by other authors \cite{CeKiNe12, CeKiNe13}, is proved to be the right notion in the sense that
\begin{equation}\label{eq2}
\mathcal{P}(D^{\alpha}_t u)(n+m) = \Delta^{\alpha} \mathcal{P}(u)(n),\quad n \in \mathbb{N}_0,
\end{equation}
where $D^{\alpha}_{t}$ denotes the Riemann-Liouville fractional derivative on $\mathbb{R}_+$ of order $\alpha >0, \,$  $m-1<\alpha \leq m$ and
$$
\Delta^{\alpha}v(n)= \Delta^m \Bigg(  \sum_{j=0}^n k^{m-\alpha}(n-j)v(j) \Bigg),
$$
where,
$$
k^{\beta}(n):= \frac{\Gamma(\beta  +n)}{\Gamma(\beta)
\Gamma(n+1)} = \int_0^{\infty} p_n(t)\frac{t^{\beta-1}}{\Gamma(\beta)}dt, \quad n \in \mathbb{N}_0, \quad \beta>0.
$$
See Definition \ref{def2.6} and Theorem \ref{th3.4} below. Then, we can show an interesting connection between the Delta operator of order $\alpha$ (i.e. the Riemann-Liouville-like fractional difference) in the right hand side of \eqref{eq2}  and the Caputo-like fractional difference by means of the identity (Theorem \ref{th4.1}):
$$
_C\Delta^{\alpha}u(n)= \Delta^{\alpha}u(n) - k^{2-\alpha}(n+1)[u(1)-2u(0)] - k^{2-\alpha}(n+2)u(0),  \quad n \in \mathbb{N}_0.
$$
In Section 3, we  use successfully the preceding definitions and properties to solve the problem \eqref{eq1}, firstly, in the homogeneous linear case. In order to do that, we construct a distinguished sequence of bounded  and linear operators $\{\mathcal{S}_{\alpha}(n)\}_{n\in \N_0}$ that solves the homogeneous linear initial value problem
\begin{equation*}
\left\{%
\begin{array}{rll}
    \Delta^{\alpha} u(n) &=  Au(n+2),  \quad n \in \mathbb{N}_0, \quad 1 < \alpha\le2;  \\
    u(0) &= u_0 \in D(A); \\ u(1) & = u_1 \in D(A).
\end{array}%
\right.
\end{equation*}
See Theorem \ref{th5.6}. In particular, when the operator $A$ is bounded, we derive an explicit representation of the solution (Proposition \ref{th4.3}).

From a different point of view, this   representation can be considered as the discrete counterpart of the  Mittag-Leffler function $t^{\alpha-1}E_{\alpha,\alpha}(At^{\alpha})$ (when $A$ is a complex number) which interpolates between the exponential and hyperbolic sine function for $1<\alpha<2.$

 In Section 4 we study the fully nonlinear problem \eqref{eq1}. After to introduce the notion of solution, which is motivated by the representation of the solution in the non-homogeneous linear case (Corollary \ref{cor4.5}), we consider a distinguished  class of vector-valued  spaces of  weighted sequences, that behaves like
$$
l^{\infty}_w(\mathbb{N}_0;X):= \left\{\xi:\mathbb{N}_0\to X \quad / \quad \sup_{n\in\mathbb{N}_0}\frac{\|\xi(n)\|}{nn!}<\infty \right\}.
$$ This vector-valued Banach spaces of sequences will play a central role in the development of this section. The main ingredient for the success of our analysis is the observation that the special weight $w(n)=nn!,$ that represents the factorial representation of a positive integer, proves to be suitable to find existence of solutions for \eqref{eq1} in the above defined space $l^{\infty}_w(\mathbb{N}_0;X)$ under the hypothesis of only boundedness of the sequence of operators $\mathcal{S}_{\alpha}(n).$  We give two positive results in this direction. See Theorem \ref{thlip} and Theorem  \ref{th5.3}.
In Section 5,  we prove several relations between the continuous and discrete setting,  including the notable identity  \eqref{eq2}. See Theorem \ref{th3.4} below. This relations are obtained in the context of the Poisson transformation \eqref{eq1.2} whose main properties are established in Theorem \ref{th3.2}. Note that the idea of discretization of the fractional derivative in time was employed  in  the paper \cite{Cu07} (see also \cite{Cu-Lu-Pa06} and references therein). Section 6 is devoted to the construction of sequences of operators $\{{\mathcal S}_{\alpha}(n)\}_{n\in \N_0}$ via subordination by the Poisson transformation of $\alpha$-resolvent families generated by $A$ (Theorem \ref{th5.66}). A remarkable consequence is Theorem \ref{cor5.4}, which proves existence of solution for the nonlinear problem \eqref{eq1} in the space $l^{\infty}_w(\mathbb{N}_0;X)$ under the hypothesis that $A$ is the generator of a bounded sine family such that the resolvent operator $(\lambda -A)^{-1}$ is a compact operator for some $\lambda$ large enough.

 Finally,  Section 7 provide us with several examples and applications of our general theorems, notably concerning the cases where either $A$ is a multiplication operator or the second order partial differential operator $\frac{\partial^2}{\partial x^2}.$ We also pay special attention to the case $\alpha=2$ and to some related problems formatted in a slighty different way than \eqref{eq1}. \\

\noindent{\bf Notation} We denote by
$
\mathbb{N}_0:=\{ 0, 1, 2, ...\},
$ the set of non-negative integer numbers and  $X$  a
complex Banach space. We denote by $s( \mathbb{N}_0; X)$ the
vectorial space consisting of all vector-valued sequences $u:\mathbb{N}_0 \to
X.$ We recall that the {\it $Z$-transform} of a  vector-valued sequence $f\in
s(\mathbb{N}_0;X)$,  is defined by
\begin{equation*}
\widetilde u(z) := \sum_{j=0}^{\infty} z^{-j} u(j)
\end{equation*}
where $z$ is a complex number. Note that convergence of the series is given for $|z|>R$ with $R$ sufficiently large.

Recall that the finite convolution $*$ of two sequences $u\in s( \mathbb{N}_0; \mathbb{C})$ and $v\in s( \mathbb{N}_0; X)$ is defined by
\begin{equation*}
(u*v)(n) := \sum_{j=0}^n u(n-j)v(j), \quad n \in \mathbb{N}_0.
\end{equation*}
It is well known that \begin{equation}\label{Ztrans}\widetilde{(u*v)}(z)=\widetilde{u}(z)\widetilde{v}(z),\quad |z|>\max\{R_1,R_2\},\end{equation} where $R_1$ and $R_2$ are the radius of convergence of the $Z$-transforms of $u$ and $v$ respectively. The Banach space $\ell^1(X)$ is the subset of $s( \mathbb{N}_0; X)$ such that
$\Vert u\Vert_1:=\sum_{n=0}^\infty \Vert u(n)\Vert<\infty; $ and the Lebesgue space $L^1(\mathbb{R}_+; X)$ is formed by measurable functions $f:\mathbb{R}_+\to X$ such that
$$
\Vert f\Vert_1:= \int_0^\infty \Vert f(t)\Vert dt\,<\infty.
$$
The usual Laplace transform is given by
$$
\hat f(\lambda):= \int_0^{\infty} e^{-\lambda t} f(t)dt,\qquad  \Re \lambda >0, \quad f\in L^1(\mathbb{R}_+; X).
$$

In the case $X= \mathbb{C}$, the Banach space $L^1(\mathbb{R}_+)$ is, in fact, a Banach algebra with the usual convolution product $\ast$ given by
$$
f\ast g(t):=\int_0^t f(t-s)g(s)ds, \qquad t\ge 0, \quad f, g\in L^1(\mathbb{R}_+).
$$
The same holds in the case of $(\ell^1, \ast)$. The Banach space $ C^{(m)}(\mathbb{R}_+;X)$ is formed for continuous functions which have $m$-continuous derivatives defined on $\mathbb{R}_+$ with $m\in \mathbb{N}_0$.

Let $S: \mathbb{R}_+ \to \mathcal{B}(X)$ be strongly continuous, that is, for all $x \in X$ the map $ t\to S(t)x$ is continuous on $\mathbb{R}_+.$ We say that a family of bounded and linear operators $\{S(t)\}_{t\geq 0}$ is exponentially bounded if there exist real numbers $M >0$ and $\omega \in \mathbb{R}$ such that
$$
\| S(t) \| \leq Me^{\omega t}, \quad t \geq 0.
$$
We say that $\{S(t)\}_{t\geq 0}$ is bounded if $\omega =0.$    Note that if $\{S(t)\}_{t\geq 0}$ is exponentially bounded then the Laplace transform
$
\hat S(\lambda)x
$
exists for all $\Re(\lambda) >\omega.$

\section{Fractional difference operators}\label{sect2}

The forward Euler operator $\Delta :s( \mathbb{N}_0; X) \to s( \mathbb{N}_0; X)$ is defined by
$$
\Delta u(n):=u(n+1)-u(n), \quad n \in \mathbb{N}_0.
$$
For $m \in \mathbb{N}$, we define recursively $\Delta^m :s(
\mathbb{N}_0; X) \to s( \mathbb{N}_0; X)$ by $\Delta^1= \Delta$ and
 $$\Delta^m:=\Delta^{m-1}\circ \Delta. $$ The operator $\Delta^m$ is called the $m$-th order forward
 difference operator and
\begin{equation}\label{msuma}
\Delta^m u(n)=\sum_{j=0}^m\binom{m}{j}(-1)^{m-j}u(n+ j), \quad n
\in \mathbb{N}_0,
\end{equation}
for  $u \in s(
 \mathbb{N}_0;X)$. We also denote by $\Delta^0 = I,$ where $I$  is the identity
operator.

We define
\begin{equation}\label{kerne}
k^{\alpha}(n) := \frac{\Gamma(\alpha +n)}{\Gamma(\alpha)
\Gamma(n+1)},  \quad n \in \mathbb{N}_0, \quad \alpha >0.
\end{equation}
This sequence has appeared in \cite{AL} and \cite{Li15} in connection with fractional difference operators. The semigroup property $k^\alpha \ast k^\beta= k^{\alpha+\beta}$  and the generating formula
\begin{equation}\label{genere}
\sum_{n=0}^\infty k^\alpha(n)z^n={1\over (1-z)^\alpha}, \qquad \vert z\vert <1,
\end{equation}
hold for $\alpha, \beta>0$, see for example \cite[Vol. I, p.77]{Zygmund}.

The following definition of fractional sum (also called Ces\`{a}ro sum in \cite{Zygmund}) has appeared recently in some papers, see for example \cite{AL, Li15}. It has proven to be useful in the treatment of fractional difference equations. Note that this definition is implicitly included in e.g. \cite{Ab-At12, AE2009b, MR1989}.

\begin{definition}\cite[Definition 2.5]{Li15} \label{def2.3}
Let $\alpha >0.$ The $\alpha$-th  fractional sum of a sequence $u: \mathbb{N}_0 \to X$ is defined as follows
\begin{equation}\label{eq2.3b}
\Delta^{-\alpha}u(n): = \sum_{j=0}^n k^{\alpha}(n-j)u(j)=(k^\alpha\ast u)(n), \quad n \in \mathbb{N}_0.
\end{equation}
\end{definition}
One of the reasons to choose this  operator in this paper is because their flexibility to be handled by means of $Z$-transform methods. Moreover, it has a better behavior for mathematical analysis  when we ask, for example, for definitions of fractional  sums and differences on subspaces of $s(\mathbb{N}_0;X)$ like e.g. $l_p$ spaces. We notice that, recently, this approach by means of the $Z$-transform has been followed by other authors, see \cite{CeKiNe12, CeKiNe13}.

The next concept is analogous to the
definition of a fractional derivative in the sense of
Riemann-Liouville, see \cite{AE2007, MR1989}. In other words,  to a given vector-valued sequence,
first fractional summation and then integer difference are applied.

\begin{definition}\cite[Definition 2.7]{Li15} \label{def2.6} Let $\alpha\in \mathbb{R}^+\backslash \mathbb{N}_0$. The fractional difference operator of order $\alpha$ in the sense of Riemann-Liouville, $\Delta^{\alpha}: s(\mathbb{N}_0;X) \to s(\mathbb{N}_0;X),$  is defined by
\begin{equation*}
\Delta^{\alpha} u(n) := \Delta^m ( \Delta^{-(m-\alpha)} u)(n), \quad n \in \mathbb{N}_0,
\end{equation*}
where $m-1 < \alpha < m.$
\end{definition}

\begin{example}\label{coro3.6} Let $1<\beta$. Then
$$
\Delta k^\beta(n)=\frac{\Gamma(\beta +n+1)}{\Gamma(\beta)
\Gamma(n+2)}-\frac{\Gamma(\beta +n)}{\Gamma(\beta)
\Gamma(n+1)}= {(\beta-1)\Gamma(\beta+n)\over \Gamma(\beta)\Gamma(n+2)}=k^{\beta-1}(n+1), \qquad   n \in \mathbb{N}_0.
$$
We iterate $m$-times  with $m\in \mathbb{N}$ to get for $\beta >m$ that
$$\Delta^m k^\beta(n)=k^{\beta-m}(n+m),  \qquad   n \in \mathbb{N}_0.
$$
Let $0<\alpha <\beta $ and $m-1<\alpha<m$ for $m\in \mathbb{N}$. By Definition \ref{def2.6}  and \eqref{def2.3}, we get that
$$
\Delta^{\alpha} k^{\beta}(n)= \Delta^m ( \Delta^{-(m-\alpha)}  k^{\beta})(n)=\Delta^m ( k^{m-\alpha}  \ast k^{\beta})(n)=\Delta^m ( k^{m-\alpha+\beta})=k^{\beta-\alpha}(n+m)  ,
$$
for $  n \in \mathbb{N}_0.$ This equality extends \cite[Corollary 3.6]{Li14} given for $0<\alpha<1.$ \end{example}

Interchanging the order of the operators in the definition of
fractional difference in the sense of Riemann-Liouville, and in
analogous way as above, we can introduce the notion of fractional
difference in the sense of Caputo as follows.

\begin{definition}\cite[Definition 2.8]{Li15}
Let $\alpha\in \mathbb{R}^+\backslash \mathbb{N}_0$. The  fractional difference operator of order $\alpha$ in the sense of Caputo,  $_C\Delta^{\alpha}: s(\mathbb{N}_0;X) \to s(\mathbb{N}_0;X),$ is defined by
\begin{equation*}
_C\Delta^{\alpha} u(n) :=
\Delta^{-(m-\alpha)}(\Delta^m u)(n), \quad n \in
\mathbb{N}_{0},
\end{equation*}
where $m-1< \alpha < m$.
\end{definition}

For further use, we note the following relation between the Caputo and Riemann-Liouville fractional differences of order $1<\alpha<2.$ The connection between the Caputo and Riemann-Liouville fractional differences of order $0<\alpha<1$ is given in \cite[Theorem 2.4]{Li14}.

\begin{thm}\label{th4.1}
For each $1< \alpha < 2$ and $u \in s(\mathbb{N}_0;X)$, we have
$$
_C\Delta^{\alpha}u(n)= \Delta^{\alpha}u(n) -  k^{2-\alpha}(n+1)[u(1)-2u(0)] -  k^{2-\alpha}(n+2)u(0), \quad n \in \mathbb{N}_0.
$$
\end{thm}
\begin{proof}
By definition and \eqref{eq2.3b} we have
$$
\begin{array}{lll}
\Delta^{-(2-\alpha)}(\Delta^2 u)(n) &=& \displaystyle \sum_{j=0}^n k^{2-\alpha}(n-j)\Delta^2 u(j)
= \displaystyle \sum_{j=0}^n k^{2-\alpha}(n-j)u(j+2) \\ &\quad &\qquad- \displaystyle 2\sum_{j=0}^n k^{2-\alpha}(n-j)u(j+1) +\sum_{j=0}^n k^{2-\alpha}(n-j)u(j)\\
&=& \displaystyle \sum_{j=2}^{n+2} k^{2-\alpha}(n+2-j)u(j)-2\sum_{j=1}^{n+1} k^{2-\alpha}(n+1-j)u(j)\\ & \quad &\qquad+ \displaystyle  \sum_{j=0}^n k^{2-\alpha}(n-j)u(j) = \displaystyle \sum_{j=0}^{n+2} k^{2-\alpha}(n+2-j)u(j) \\ &\quad &\qquad- \displaystyle 2\sum_{j=0}^{n+1} k^{2-\alpha}(n+1-j)u(j)+\sum_{j=0}^n k^{2-\alpha}(n-j)u(j) \\ \\
&\quad& \qquad - k^{2-\alpha}(n+2)u(0)-k^{2-\alpha}(n+1)u(1)+2k^{2-\alpha}(n+1)u(0)\\ \\
&=& \displaystyle \Delta^2(\Delta^{-(2-\alpha)} u)(n)- k^{2-\alpha}(n+1)(u(1)-2u(0)) - k^{2-\alpha}(n+2)u(0),
\end{array}
$$
and so we obtain the desired result.
\end{proof}

We also have the following property for the Riemann-Liouville fractional difference of the convolution.

\begin{thm} \label{deltaconv}
Let $1< \alpha \le 2$, $u \in s(\mathbb{N}_0;\mathbb{C})$ and $v \in s(\mathbb{N}_0;X)$. Then, for each $n\in\mathbb{N}_0$ the following identity holds
$$\Delta^\alpha(u*v)(n)=(\Delta^\alpha u*v)(n)+ (u(1)-\alpha u(0))v(n+1)+ u(0)v(n+2).$$
\end{thm}

\begin{proof} For each $n\in\mathbb{N}_0,$
\begin{eqnarray*}
\Delta^\alpha(u*v)(n)
&=& \Delta^{-(2-\alpha)}(u*v)(n+2) -2 \Delta^{-(2-\alpha)}(u*v)(n+1) + \Delta^{-(2-\alpha)}(u*v)(n)\\
&=&  \sum_{j=0}^{n+2}(k^{2-\alpha}*u)(n+2-j)v(j) - 2\sum_{j=0}^{n+1}(k^{2-\alpha}*u)(n+1-j)v(j) \\
&\,&+ \sum_{j=0}^{n}(k^{2-\alpha}*u)(n-j)v(j) \\
&=& \sum_{j=0}^{n}(k^{2-\alpha}*u)(n+2-j)v(j)
-2\sum_{j=0}^{n}(k^{2-\alpha}*u)(n+1-j)v(j) \\ &\,&+ \sum_{j=0}^{n}(k^{2-\alpha}*u)(n-j)v(j)
+ (k^{2-\alpha}*u)(1)v(n+1) \\ &\, & +(k^{2-\alpha}*u)(0)v(n+2)-2(k^{2-\alpha}*u)(0)v(n+1)\\
&=& \sum_{j=0}^{n}\Delta^2(k^{2-\alpha}*u)(n-j)v(j) + (k^{2-\alpha}(0)u(1)+k^{2-\alpha}(1)u(0))v(n+1) \\
&\,&+ (k^{2-\alpha}(0)u(0))v(n+2)-2(k^{2-\alpha}(0)u(0))v(n+1)
\\ &=& \sum_{j=0}^{n}\Delta^\alpha u(n-j)v(j)
+ (u(1)+(2-\alpha)u(0))v(n+1)+u(0)v(n+2) \\ &\,&- 2u(0)v(n+1)\\
&=& (\Delta^\alpha u*v)(n) + (u(1)-\alpha u(0))v(n+1)+ u(0)v(n+2),
\end{eqnarray*}
proving the claim.
\end{proof}
We notice that for $0<\alpha\leq 1$ the above property reads
$$
\Delta^\alpha(u*v)(n)= (\Delta^{\alpha}u*v)(n) + u(0)v(n+1), \quad n \in \mathbb{N}_0.
$$
It has been proved only recently in \cite[Lemma 3.6]{Li15}.
\section{Linear fractional difference equations on Banach spaces}

Let $A$ be a closed linear operator defined on a Banach space $X$. In this section we study the problem
\begin{equation}\label{eq4.1}
\left\{%
\begin{array}{rll}
    \Delta^{\alpha} u(n) &=  Au(n+2) + f(n),  \quad n \in \mathbb{N}_0, \quad 1<\alpha\le2;  \\
    u(0) &= u_0 ;\\ u(1) & =u_1.
\end{array}%
\right.
\end{equation}

We say that a vector valued sequence $u\in s(\mathbb{N}_0;X)$ is a  solution of \eqref{eq4.1} if $u(n) \in D(A)$ for all $n \in \mathbb{N}_0$ and $u$ satisfies \eqref{eq4.1}.

We will use the notion of discrete $\alpha$-resolvent family introduced in \cite[Definition 3.1]{AL} to obtain the solution of the problem (\ref{eq4.1}). Note that the knowledge of the abstract properties of this family of bounded operators provide insights on the qualitative behavior of the solutions of fractional difference equations.

\begin{definition}\label{DefResol} Let $\alpha>0$ and $A$ be a closed linear operator with domain $D(A)$ defined on a Banach space $X.$ An operator-valued sequence $\{{\mathcal S}_{\alpha}(n)\}_{n\in\N_0}\subset \mathcal{B}(X)$ is called a discrete $\alpha$-resolvent family generated by $A$ if it satisfies the following conditions\begin{itemize}
\item[(i)] ${\mathcal S}_{\alpha}(n)Ax=A{\mathcal S}_{\alpha}(n)x$ for $n\in\N_0$ and $x\in D(A);$
\item[(ii)] ${\mathcal S}_{\alpha}(n)x=k^{\alpha}(n)x+A(k^{\alpha}*{\mathcal S}_{\alpha})(n)x,$ for all $n\in\N_0$ and $x\in X.$
\end{itemize}
The family  $\{{\mathcal S}_{\alpha}(n)\}_{n\in\N_0}$ is said bounded if $\Vert{\mathcal S}\Vert_\infty:= \sup_{n\in\N_0 }\Vert{\mathcal S}_{\alpha}(n) \Vert<\infty$.
\end{definition}

An explicit representation of discrete $\alpha$-resolvent family generated by  bounded operators $A$ with $\Vert A\vert <1$  is given in the following proposition.
\begin{prop}\label{th4.3}
Let  $\alpha>0$ and $A\in\mathcal{B}(X)$, with $\|A\|<1$. Then the operator $A$ generates a discrete $\alpha$-resolvent family $\{{\mathcal S}_{\alpha}(n)\}_{n\in\N_0}$ given by
$$
\mathcal{S}_\alpha(n)=\sum_{j=0}^\infty k^{\alpha(j+1)}(n)A^j, \qquad n\in\mathbb{N}_0.
$$
\end{prop}

\begin{proof}
Since $k^{\alpha}(n)=\frac{n^{\alpha-1}}{\Gamma(\alpha)}(1+O({1\over n})),$ for $  n\in \N,$ (see for example
\cite[Vol. I, (1.18)]{Zygmund}), then the  series is convergent for $\|A\|<1$.
Take $x\in X$ and $n \in \mathbb{N}_0$. Then we get that  \begin{eqnarray*}
A(k^{\alpha}*\mathcal{S}_\alpha)(n)x&=&A\displaystyle\sum_{j=0}^n k^{\alpha}(n-j)\mathcal{S}_\alpha(j)x=A\displaystyle\sum_{j=0}^n k^{\alpha}(n-j)\sum_{i=0}^{\infty}k^{\alpha(i+1)}(j)A^ix \\
&=&\sum_{i=0}^{\infty}A^{i+1}x\displaystyle\sum_{j=0}^n k^{\alpha}(n-j)k^{\alpha(i+1)}(j) =\sum_{i=0}^{\infty}k^{\alpha(i+2)}(n)A^{i+1}x,
\end{eqnarray*}  where we have applied the semigroup property of the kernel $k^{\alpha}.$ Then we obtain $$k^{\alpha}(n)x+A(k^{\alpha}*\mathcal{S}_\alpha)(n)x=\sum_{i=0}^{\infty}k^{\alpha(i+1)}(n)A^{i}x=\mathcal{S}_\alpha(n)x,\qquad n\in \mathbb{N}_0,$$
and we conclude the proof.
\end{proof}


For $\alpha >0$ fixed and each $n \in \mathbb{N}$  the sequence $\{\beta_{\alpha,n}(j)\}_{j=1,...,n}$ was introduced in \cite[Section 3.1]{AL} as follows:

For $n=1$, $$\beta_{\alpha,1}(1)=k^{\alpha}(1)=\alpha.$$

For $n=2,$
\begin{eqnarray*}\beta_{\alpha,2}(1)&=&k^{\alpha}(2)-k^{\alpha}(1)\beta_{\alpha,1}(1)=k^{\alpha}(2)-\left(k^{\alpha}(1)\right)^2,\cr
\beta_{\alpha,2}(2)&=&k^{\alpha}(1)\beta_{\alpha,1}(1)=\left(k^{\alpha}(1)\right)^2=\alpha^2.
\end{eqnarray*}

For $n=3,$ \begin{eqnarray*} \beta_{\alpha,3}(1)&=&k^{\alpha}(3)-k^{\alpha}(2)\beta_{\alpha,1}(1)-k^{\alpha}(1)\beta_{\alpha,2}(1)=k^{\alpha}(3)-2k^{\alpha}(2)k^{\alpha}(1)+(k^\alpha(1))^3,\\
\beta_{\alpha,3}(2)&=&k^{\alpha}(2)\beta_{\alpha,1}(1)+k^{\alpha}(1)\beta_{\alpha,2}(1)-k^{\alpha}(1)\beta_{\alpha,2}(2)=2 k^\alpha(2)k^\alpha(1)-2\left(k^{\alpha}(1)\right)^3,
\\
\beta_{\alpha,3}(3)&=&k^{\alpha}(1)\beta_{\alpha,2}(2)= \left(k^{\alpha}(1)\right)^3=\alpha^3.
\end{eqnarray*}

For $n\geq 4,$
\begin{eqnarray*}
\beta_{\alpha,n}(1)&=&k^{\alpha}(n)-\displaystyle\sum_{j=1}^{n-1}k^{\alpha}(n-j)\beta_{\alpha,j}(1),\\
\beta_{\alpha,n}(l)&=&\displaystyle\sum_{j=l-1}^{n-1}k^{\alpha}(n-j)\beta_{\alpha,j}(l-1)-\displaystyle\sum_{j=l}^{n-1}k^{\alpha}(n-j)\beta_{\alpha,j}(l) \quad \hbox{for} \quad 2\leq l\leq n-1,\\
\beta_{\alpha,n}(n)&=&k^{\alpha}(1)\beta_{\alpha,n-1}(n-1)= \left(k^{\alpha}(1)\right)^n=\alpha^n
\end{eqnarray*}

In case that $A$ is  closed, but not necessarily bounded, the authors in \cite[Theorem 3.2]{AL}  proved that given  $\{{\mathcal S}_\alpha(n)\}_{n\in\N_0}\subset \mathcal{B}(X)$  a discrete $\alpha$-resolvent family generated by $A,$ then $1\in\rho(A)$ and  ${\mathcal S}_\alpha(0)=(I-A)^{-1}$; ${\mathcal S}_\alpha(0)x\in D(A)$ and ${\mathcal S}_\alpha(n)x\in D(A^2)$ for all $n\in\N,$ and $x\in X$;
and \begin{equation*}{\mathcal S}_\alpha(n)x=\displaystyle\sum_{j=1}^{n}\beta_{\alpha,n}(j)(I-A)^{-(j+1)}x,\quad n\in\N
, \quad x\in X.
\end{equation*}
The last equality  provides an explicit representation of discrete $\alpha$-resolvent families  in terms of a bounded linear operators which is, in fact, a characterization of this family of operators as the next theorem shows.

\begin{thm} \label{Caract} Let $\lambda,\alpha >0$, $(A, D(A))$ be a closed operator on the Banach space $X$ and $\{{\mathcal S}_\alpha(n)\}_{n\in\N_0}\subset \mathcal{B}(X)$ be a sequence of bounded operators.
Then the following conditions are equivalent.
\begin{itemize}
\item[(i)]  The family $\{{\mathcal S}_\alpha(n)\}_{n\in\N_0}\subset \mathcal{B}(X)$ is  a discrete $\alpha$-resolvent family generated by $A.$
\item[(ii)] $1\in\rho(A)$, the operator ${\mathcal S}_\alpha(0)=(I-A)^{-1}$
and \begin{equation*}{\mathcal S}_\alpha(n)x=\displaystyle\sum_{j=1}^{n}\beta_{\alpha,n}(j)(I-A)^{-(j+1)}x,\quad n\in\N
, \quad x\in X.
\end{equation*}
\end{itemize}

If there exists $\lambda_{0}>0$ such that $\sup_{n\in {\N_0}}\lambda_{0}^{-n}\Vert {\mathcal S}_\alpha(n)\Vert <\infty,$ both equations are equivalent to
\begin{itemize}
\item[(iii)] $\displaystyle{\left(\lambda-1\over \lambda\right)^\alpha}\in \rho(A)$ and
\begin{equation}\label{3.2}
\left(\left(\lambda-1\over \lambda\right)^\alpha-A\right)^{-1}x=  \sum_{n=0}^\infty \lambda^{-n}{\mathcal S}_\alpha(n)x, \qquad x\in X, |\lambda|>\max\{\lambda_0,1\}.
\end{equation}
\end{itemize}
\end{thm}

\begin{proof} The condition (i) implies the condition (ii) is given in \cite[Theorem 3.2]{AL}. Now we suppose that the condition (ii) holds. Then  ${\mathcal S}_\alpha(n)x\in D(A)$ for any $x\in X$ and $n\in {\N_0}.$ For $n\in\N$ and $x\in X$ we have that \begin{eqnarray*}
(I-A)\mathcal{S}_{\alpha}(n)x&=&\displaystyle\sum_{j=1}^{n}\beta_{\alpha,n}(j)(I-A)^{-j}x=(k^{\alpha}(n)-\sum_{i=1}^{n-1}k^{\alpha}(n-i)\beta_{\alpha,i}(1))(I-A)^{-1}x \\
&&+\displaystyle\sum_{j=2}^{n-1}(\displaystyle\sum_{i=j-1}^{n-1}k^{\alpha}(n-i)\beta_{\alpha,i}(j-1)-\displaystyle\sum_{i=j}^{n-1}k^{\alpha}(n-i)\beta_{\alpha,i}(j))(I-A)^{-j}x \\
&&+k^{\alpha}(1)\beta_{\alpha,n-1}(n-1)(I-A)^{-n}x\\
&=&k^{\alpha}(n)(I-A)^{-1}x+\displaystyle\sum_{j=1}^{n-1}\sum_{i=j}^{n-1} k^{\alpha}(n-i) \beta_{\alpha,i}(j)((I-A)^{-(j+1)}-(I-A)^{-j})x.
\end{eqnarray*}
Applying that $(I-A)^{-1}-I=A(I-A)^{-1}$ and $\mathcal{S}_{\alpha}(0)=(I-A)^{-1}$ we get that\begin{eqnarray*}
(I-A)\mathcal{S}_{\alpha}(n)x&=&k^{\alpha}(n)(I+A\mathcal{S}_{\alpha}(0))x +A\displaystyle\sum_{i=1}^{n-1} k^{\alpha}(n-i) \sum_{j=1}^{i} \beta_{\alpha,i}(j)(I-A)^{-(j+1)}x \\
&=&k^{\alpha}(n)(I+A\mathcal{S}_{\alpha}(0))x + A\displaystyle\sum_{i=1}^{n-1} k^{\alpha}(n-i) \mathcal{S}_{\alpha}(i)x,
\end{eqnarray*}
and clearly it follows that $\mathcal{S}_{\alpha}(n)x=k^{\alpha}(n)x+A(k^{\alpha}*{\mathcal S}_{\alpha})(n)x$ for $n\in\N.$ The case $n=0$ is a simple check.

Finally we prove that if there exists $\lambda_{0}>0$ such that $\sup_{n\in {\N_0}}\lambda_{0}^{-n}\Vert {\mathcal S}_\alpha(n)\Vert <\infty,$ (iii) is equivalent to (i). Assume that $\{{\mathcal S}_\alpha(n)\}_{n\in\N_0}\subset \mathcal{B}(X)$ is  a discrete $\alpha$-resolvent family generated by $A,$ then applying $Z$-transform we get for $|\lambda|>\max\{\lambda_0,1\}$ that \begin{eqnarray*}\widetilde{\mathcal{S}_{\alpha}} (\lambda)x &=& \sum_{j=0}^{\infty} \lambda^{-j} \mathcal{S}_{\alpha}(j)x=\widetilde{k^{\alpha}}(\lambda)x+A\widetilde{k^{\alpha}}(\lambda)\widetilde{\mathcal{S}_{\alpha}} (\lambda)x\\
&=&\biggl(\frac{\lambda}{\lambda-1}\biggr)^{\alpha}x+\biggl(\frac{\lambda}{\lambda-1}\biggr)^{\alpha}A\widetilde{\mathcal{S}_{\alpha}} (\lambda)x,\quad x\in X,\end{eqnarray*} and $$\widetilde{\mathcal{S}_{\alpha}} (\lambda)x=\biggl(\frac{\lambda}{\lambda-1}\biggr)^{\alpha}x+\biggl(\frac{\lambda}{\lambda-1}\biggr)^{\alpha}\widetilde{\mathcal{S}_{\alpha}} (\lambda)Ax,\quad x\in D(A),$$ where we have used that \eqref{Ztrans} and \eqref{genere}. Thus the operator $\left(\lambda-1\over \lambda\right)^\alpha-A$ is invertible, and we get \eqref{3.2}. Conversely, let $|\lambda|,|\mu|>\max\{\lambda_0,1\}$ and $x\in D(A),$ then there exists $y\in X$ such that $x=\left(\left(\mu-1\over \mu\right)^\alpha-A\right)^{-1}y.$ Using that $\left(\left(\lambda-1\over \lambda\right)^\alpha-A\right)^{-1}$ and $\left(\left(\mu-1\over \mu\right)^\alpha-A\right)^{-1}$ are bounded operators and commute, and $A$ is closed we have that \begin{eqnarray*}
\widetilde{\mathcal{S}_{\alpha}} (\lambda)x&=&\widetilde{\mathcal{S}_{\alpha}}(\lambda)\left(\left(\mu-1\over \mu\right)^\alpha-A\right)^{-1}y \\
&=&\left(\left(\mu-1\over \mu\right)^\alpha-A\right)^{-1}\left(\left(\lambda-1\over \lambda\right)^\alpha-A\right)^{-1}y \\
&=&\sum_{n=0}^\infty \lambda^{-n}\left(\left(\mu-1\over \mu\right)^\alpha-A\right)^{-1}\mathcal{S}_{\alpha}(n)\left(\left(\mu-1\over \mu\right)^\alpha-A\right)x.
\end{eqnarray*}
The uniqueness of $Z$-transform proves that $$\mathcal{S}_{\alpha}(n)x=\left(\left(\mu-1\over \mu\right)^\alpha-A\right)^{-1}\mathcal{S}_{\alpha}(n)\left(\left(\mu-1\over \mu\right)^\alpha-A\right)x.$$ Then we have $\mathcal{S}_{\alpha}(n)x\in D(A),$ and therefore $A\mathcal{S}_{\alpha}(n)x=\mathcal{S}_{\alpha}(n)Ax$ for all $x\in X.$ Finally note that for $|\lambda|>\max\{\lambda_0,1\}$ and $x\in D(A)$ we have using \eqref{genere} that \begin{eqnarray*}
\widetilde{k^{\alpha}}(\lambda)x&=&\widetilde{k^{\alpha}}(\lambda)\widetilde{\mathcal{S}_{\alpha}}(\lambda)\left(\left(\lambda-1\over \lambda\right)^\alpha-A\right)x \\
&=&\widetilde{\mathcal{S}_{\alpha}}x-\widetilde{(k^{\alpha}*\mathcal{S}_{\alpha})}(\lambda)Ax,
\end{eqnarray*}
and by the uniqueness of $Z$-transform we get the result.
\end{proof}

A beautiful consequence of Theorem  \ref{Caract} is the following result about sums of combinatorial numbers which seems to be new.

\begin{corollary} Take $\alpha >0$, $n\in \mathbb{N}$ and $\{\beta_{\alpha,n}(j)\}_{j=1,...,n}$  defined as above. Then \begin{itemize}
\item[(i)]$\displaystyle\sum_{j=1}^n {\beta_{\alpha, n}(j)\over  (1-\lambda)^{j+1}}=\sum_{l=0}^{\infty}\lambda^l\frac{\Gamma(\alpha(l+1)+n)}{\Gamma(\alpha(l+1))\Gamma(n+1)},\quad \text{for }\vert \lambda\vert <1.$
\item[(ii)]$\displaystyle\sum_{j=1}^n \beta_{\alpha, n}(j)\frac{\Gamma(l+1+j)}{\Gamma(l+1)\Gamma(j+1)}=\frac{\Gamma(\alpha(l+1)+n)}{\Gamma(\alpha(l+1))\Gamma(n+1)},\quad \text{for } l\in\N.$
\end{itemize}
\end{corollary}
\begin{proof}
(i) We take $|\lambda|<1,$ then using Proposition \ref{th4.3} and Theorem \ref{Caract} in the scalar case we have that $$\displaystyle\sum_{j=1}^n {\beta_{\alpha, n}(j)\over  (1-\lambda)^{j+1}}=\sum_{l=0}^\infty{\lambda^l}k^{\alpha(l+1)}(n)=\sum_{l=0}^\infty{\lambda^l}\frac{\Gamma(\alpha(l+1)+n)}{\Gamma(\alpha(l+1))\Gamma(n+1)}, \qquad n\in \N.$$ (ii) Let $|\lambda|<1,$ then $$\displaystyle\sum_{j=1}^n {\beta_{\alpha, n}(j)\over  (1-\lambda)^{j+1}}=\sum_{l=0}^{\infty}\lambda^l\sum_{j=1}^n \beta_{\alpha, n}(j)\frac{\Gamma(l+1+j)}{\Gamma(l+1)\Gamma(j+1)},\qquad n\in \N,$$ where we have applied that $\displaystyle{\frac{1}{(1-\lambda)^{j+1}}=\sum_{l=0}^{\infty}\lambda^l \frac{\Gamma(l+1+j)}{\Gamma(l+1)\Gamma(j+1)}.}$ Then we apply (i) to get the result.
\end{proof}

Our main result in this section is the following theorem.

\begin{thm}\label{th5.6} Suppose that $A$ is the generator of   a discrete $\alpha$-resolvent family $\{{\mathcal S}_\alpha(n)\}_{n\in\N_0}$ on a Banach space $X.$ Then
the fractional difference equation
\begin{equation}\label{eq6.3b}
\Delta^{\alpha} u(n) =  Au(n+2), \quad n \in \mathbb{N}_0, \,\, 1 < \alpha \le2
\end{equation}
with initial conditions $u(0)=u_0 \in D(A)$ and $u(1)=u_1 \in D(A) $ admits the unique solution
\begin{equation*}\label{eq4.3}
u(n) = \mathcal{S}_{\alpha}(n)(I-A)u(0) - \alpha \mathcal{S}_{\alpha}(n-1)u(0) + \mathcal{S}_{\alpha}(n-1)(I-A)u(1), \,\, n \in \mathbb{N}_0.
\end{equation*}
\end{thm}
\begin{proof}
 Convolving the identity given in Definition \ref{DefResol}(ii) by $k^{2-\alpha}$, we obtain
$$
(k^{2-\alpha}*\mathcal{S}_{\alpha})(n)x = (k^{2-\alpha}*k^{\alpha})(n)x + A(k^{2-\alpha}*k^{\alpha}*\mathcal{S}_{\alpha})(n)x, \quad n \in \mathbb{N}_0.
$$
Using the semigroup property for the kernels $k^{\alpha}$ we have
$$
(k^{2-\alpha}*\mathcal{S}_{\alpha})(n)x = k^{2}(n)x + A(k^{2}*\mathcal{S}_{\alpha})(n)x, \quad n \in \mathbb{N}_0.
$$
This is equivalent, by definition of fractional sum and convolution, to the following identity
$$
\Delta^{-(2-\alpha)} \mathcal{S}_{\alpha}(n)x = k^{2}(n)x + A \sum_{j=0}^n k^2(n-j)\mathcal{S}_{\alpha}(j)x, \quad n \in \mathbb{N}_0.
$$
Therefore, we get using $\Delta^2 k^2(j)=0$ for $j\in\mathbb{N}_0$ that
\begin{align*}
\Delta^2 \circ \Delta^{-(2-\alpha)}\mathcal{S}_{\alpha}(n)x & = \Delta^2 k^{2}(n)x + A \Delta^2 \sum_{j=0}^{n}k^2(n-j)\mathcal{S}_{\alpha}(j)x \\ &= A \Big[\sum_{j=0}^{n+2}k^2(j)\mathcal{S}_{\alpha}(n+2-j)x -2\sum_{j=0}^{n+1}k^2(j)\mathcal{S}_{\alpha}(n+1-j)x \\ & \quad+ \sum_{j=0}^{n}k^2(j)\mathcal{S}_{\alpha}(n-j)x   \Big] \\ &= A \Big[\sum_{j=2}^{n+2}k^2(j)\mathcal{S}_{\alpha}(n+2-j)x -2\sum_{j=1}^{n+1}k^2(j)\mathcal{S}_{\alpha}(n+1-j)x \\ &\quad + \sum_{j=0}^{n}k^2(j)\mathcal{S}_{\alpha}(n-j)x +\mathcal{S}_{\alpha}(n+2)k^2(0)x \\ & \quad+ \mathcal{S}_{\alpha}(n+1)k^2(1)x  -2\mathcal{S}_{\alpha}(n+1)k^2(0)x   \Big] \\ &= A\Big[\sum_{j=0}^{n}k^2(j+2)\mathcal{S}_{\alpha}(n-j)x -2\sum_{j=0}^{n}k^2(j+1)\mathcal{S}_{\alpha}(n-j)x \\ & \quad+ \sum_{j=0}^{n}k^2(j)\mathcal{S}_{\alpha}(n-j)x + \mathcal{S}_{\alpha}(n+2)x   \Big]
\end{align*}
for all $ n \in \mathbb{N}_0.$ We note that the left hand side in the above identity corresponds to the fractional difference of order $\alpha \in (0,2)$ in the sense of Riemann-Liouville. Therefore, we obtain
\begin{equation}\label{eq4.4}
\Delta^{\alpha}\mathcal{S}_{\alpha}(n)x = A\mathcal{S}_{\alpha}(n+2)x,
\end{equation}
for all $n \in \mathbb{N}_0$ and all $x \in X.$ Define $u(n)$ as
$$
u(n)= \mathcal{S}_{\alpha}(n)(I-A)u(0) - \alpha \mathcal{S}_{\alpha}(n-1)u(0) + \mathcal{S}_{\alpha}(n-1)(I-A)u(1), \quad  n \in \mathbb{N}_0.$$ It then follows from \eqref{eq4.4} that $u$ solves
\eqref{eq6.3b}. Finally, from the identities
\begin{equation*}
 \mathcal{S}_{\alpha}(0)x = k^{\alpha}(0) x + A(k^{\alpha}* \mathcal{S}_{\alpha})(0)x = x + Ak^{\alpha}(0) \mathcal{S}_{\alpha}(0)x= x + A \mathcal{S}_{\alpha}(0)x
\end{equation*}
and
\begin{equation*}
 \mathcal{S}_{\alpha}(1)x= k^{\alpha}(1)x + A(k^{\alpha}* \mathcal{S}_{\alpha})(1)x = \alpha \mathcal{S}_{\alpha}(0)x + A \mathcal{S}_{\alpha}(1)x
\end{equation*}
which follow from Definition \ref{DefResol} (ii), we obtain $u(0)=  \mathcal{S}_{\alpha}(0)(I-A)u_0 =u_0$ and $u(1)= \mathcal{S}_{\alpha}(1)(I-A)u_0 - \alpha  \mathcal{S}_{\alpha}(0)u_0 + \mathcal{S}_{\alpha}(0)(I-A)u_1 = u_1,$ and we conclude the proof.
\end{proof}

In the non homogeneous case, we derive the following result.
\begin{corollary}\label{cor4.5}
Suppose that $A$ is the generator of   a discrete $\alpha$-resolvent family $\{{\mathcal S}_\alpha(n)\}_{n\in\N_0}$ on a Banach space $X$ and  $f$ be a vector-valued sequence.
The fractional difference equation
\begin{equation}\label{eq6.3bf}
\Delta^{\alpha} u(n) =  Au(n+2)+f(n), \quad n \in \mathbb{N}_0, \,\, 1 < \alpha \le2,
\end{equation}
 with initial conditions $u(0)=u_0 \in D(A)$ and $u(1)=u_1 \in D(A)$, admits the unique solution
\begin{equation*}\label{eq4.3f}
u(n) = \mathcal{S}_{\alpha}(n)(I-A)u(0) - \alpha \mathcal{S}_{\alpha}(n-1)u(0) + \mathcal{S}_{\alpha}(n-1)(I-A)u(1)+(\mathcal{S}_{\alpha}*f)(n-2), \,\, \end{equation*}
for all $ n \geq 2.
$
\end{corollary}

\begin{proof} Indeed, by Theorem \ref{th5.6} and Theorem \ref{deltaconv} we have $u(n) \in D(A)$  for all $n \geq 2$ and
\begin{eqnarray*}
\Delta^\alpha u(n) &=& \Delta^\alpha (\mathcal{S}_{\alpha}(n)(I-A)u(0) - \alpha \mathcal{S}_{\alpha}(n-1)u(0) + \mathcal{S}_{\alpha}(n-1)(I-A)u(1)) \\
&&+ \Delta^\alpha (\mathcal{S}_{\alpha}*f)(n-2)\\
&=& A (\mathcal{S}_{\alpha}(n+2)(I-A)u(0) - \alpha \mathcal{S}_{\alpha}(n+1)u(0) + \mathcal{S}_{\alpha}(n)(I-A)u(1))\\
&&+ (\Delta^\alpha\mathcal{S}_{\alpha}*f)(n-2) + (\mathcal{S}_{\alpha}(1)-\alpha\mathcal{S}_{\alpha}(0))f(n-1)+\mathcal{S}_{\alpha}(0)f(n)\\
&=& Au(n+2)-(A\mathcal{S}_{\alpha}*f)(n) + (\Delta^{\alpha}\mathcal{S}_{\alpha}*f)(n-2) \\
&&+ (\mathcal{S}_{\alpha}(1)-\alpha\mathcal{S}_{\alpha}(0))f(n-1)+\mathcal{S}_{\alpha}(0)f(n).
\end{eqnarray*}
From \eqref{eq4.4} it follows $(\Delta^{\alpha}\mathcal{S}_{\alpha}*f)(n-2)=(A\mathcal{S}_{\alpha}*f)(n)-A\mathcal{S}_{\alpha}(1)(n-1)-A\mathcal{S}_{\alpha}(0)(n)$, and hence we obtain
\begin{eqnarray*}
\Delta^\alpha u(n) &=& Au(n+2) + (I-A)\mathcal{S}_{\alpha}(0)f(n) + ((I-A)\mathcal{S}_{\alpha}(1)-\alpha\mathcal{S}_{\alpha}(0))f(n-1)\\
&=& Au(n+2) + f(n),
\end{eqnarray*}
where we have used that $\mathcal{S}_{\alpha}(0)=(I-A)^{-1}$ and $(I+A)\mathcal{S}_{\alpha}(1)=\alpha\mathcal{S}_{\alpha}(0)$. For $n=0$ and $n=1$ it is a simple check, using the same above arguments, that $u$ is solution of \eqref{eq6.3bf}.
\end{proof}

\section{Non-linear fractional difference equations on Banach spaces}

Let $A$ be a closed linear operator defined on a Banach space $X$. In this section we study the non linear problem
\begin{equation}\label{eq5.1}
\left\{%
\begin{array}{rll}
    \Delta^{\alpha} u(n) &=  Au(n+2) + f(n, u(n)),  \quad n \in \mathbb{N}_0, \, \,  1<\alpha\le2;  \\
    u(0) &=0 ;\\ u(1) & =0.
\end{array}%
\right.
\end{equation}
The following definition is motivated by Corollary \ref{cor4.5}. In particular, it shows their consistence with the problem \eqref{eq5.1}.
\begin{definition} Under the assumption that the operator $A$ is the generator of a  discrete $\alpha$-resolvent family $\{{\mathcal S}_\alpha(n)\}_{n\in\N_0}$ on a Banach space $X$, we say that $u:\mathbb{N}_0 \to X$ is a solution of the non-linear problem \eqref{eq5.1} if $u$ satisfies
\begin{equation*}\label{eq5.2}
u(n) = \sum_{k=0}^{n-2}\mathcal{S}_\alpha(n-2-k)f(k,u(k)), \,\, n=2,3,4,...
\end{equation*}
\end{definition}
The next concept of admissibility is one of the keys ingredients for the estimates that we will use in the proofs of our main results on existence of solutions to \eqref{eq5.1}.
\begin{definition} We say that  a sequence $h:\mathbb{N}_0\to (0,\infty)$ is an admissible weight if $$\displaystyle\lim_{n\to\infty}\frac{1}{h(n)}\sum_{k=0}^{n-2}h(k)=0.$$
\end{definition}

\begin{example}\label{example5.3}  The sequence $h(n)=nn!,$ that represents the factorial number system, is an admissible weight function, since by \cite[formula 33 p.598]{pbm},  we have $\displaystyle\sum_{k=1}^n kk!=(n+1)!-1.$ \end{example}

For each admissible weight sequence $h$, we consider the vector-valued weighted space
$$ l_h^\infty(\mathbb{N}_0;X)=\left\{\xi:\mathbb{N}_0\to X\,\, \vert \,\,\|\xi\|_h<\infty \right\}, $$
where the norm $\|\quad\|_h$ is defined by $\|\xi\|_h:=\displaystyle\sup_{n\in\mathbb{N}_0}\frac{\|\xi(n)\|}{h(n)}$.

The following  is our first positive result on existence of solutions for the problem \eqref{eq5.1}. It  uses a Lipschitz type condition.

\begin{thm} \label{thlip} Let $h$ be an admissible weight and define $$H:=\displaystyle\sup_{n\in\mathbb{N}_0}\frac{1}{h(n)}\sum_{k=0}^{n-2}h(k).$$ Let $A$ be the generator of a bounded discrete $\alpha$-resolvent family $\{{\mathcal S}_\alpha(n)\}_{n\in\N_0}$ on a Banach space $X$ for some $1<\alpha\le2$, and  let $f:\mathbb{N}_0\times X\to X$ be such that $f(k,0)=0$ for all $k\in\mathbb{N}_0$, verifying the following hypothesis:
\begin{enumerate}
\item[(L)] The function $f$ satisfies a Lipschitz condition in $x\in X$ uniformly in $k\in\mathbb{N}_0$, that is, there exists a constant $L>0$ such that $\|f(k,x)-f(k,y)\|\le L\|x-y\|$, for all $x,y\in X$, $k\in\mathbb{N}_0$, with $L<(\|\mathcal{S}_\alpha\|_\infty H)^{-1}$.
\end{enumerate}
Then the problem \eqref{eq5.1} has an unique solution in $l_h^\infty(\mathbb{N};X)$.
\end{thm}

\begin{proof} Let us define the operator $G:l_h^\infty(\mathbb{N}_0;X)\to l_h^\infty(\mathbb{N}_0;X)$ given by
$$
Gu(n)=\sum_{k=0}^{n-2}\mathcal{S}_\alpha(n-2-k)f(k,u(k)), \quad n\ge2.
$$
First, we show that $G$ is well defined: Let $u\in l_h^\infty(\mathbb{N}_0;X)$ be given. By using the assumption (L) for $y=0$ and the boundedness of $\{{\mathcal S}_\alpha(n)\}_{n\in\N_0}$ we get that,
$$
\|Gu(n)\| \le \sum_{k=0}^{n-2}\|\mathcal{S}_\alpha(n-2-k)\|\|f(k,u(k))\|
\le \|\mathcal{S}_\alpha\|_\infty L\sum_{k=0}^{n-2}\|u(k)\|
\le \|\mathcal{S}_\alpha\|_\infty L\|u\|_h\sum_{k=0}^{n-2}h(k),
$$
 for each $n\in\mathbb{N}_0.$ Hence,
$$
\frac{\|Gu(n)\|}{h(n)} \le \|\mathcal{S}_\alpha\|_\infty L\|u\|_h\frac{1}{h(n)}\sum_{k=0}^{n-2}h(k).
$$
It proves that $Gu\in l_h^\infty(\mathbb{N}_0;X)$.
We next prove that $G$ is a contraction on $l_h^\infty(\mathbb{N}_0;X).$ Indeed, let $u,v\in l_h^\infty(\mathbb{N}_0;X)$ be given. Then, for each $n\in\mathbb{N}_0$,
\begin{eqnarray*}
\|Gu(n)-Gv(n)\|&\le&\sum_{k=0}^{n-2}\|\mathcal{S}_\alpha(n-2-k)\|\|f(k,u(k))-f(k,v(k))\|\\
&\le& \|\mathcal{S}_\alpha\|_\infty\sum_{k=0}^{n-2}\|f(k,u(k))-f(k,v(k))\|\\
&\le& \|\mathcal{S}_\alpha\|_\infty\sum_{k=0}^{n-2}L\|u(k)-v(k)\|
\le \|\mathcal{S}_\alpha\|_\infty L\|u-v\|_h\sum_{k=0}^{n-2}h(k),
\end{eqnarray*}
where we have used the assumption (L). Therefore
$$ \frac{\|Gu(n)-Gv(n)\|}{h(n)} \le \|\mathcal{S}_\alpha\|_\infty L\|u-v\|_h\frac{1}{h(n)}\sum_{k=0}^{n-2}h(k),$$
and consequently
$$ \|Gu-Gv\|_h \le \|\mathcal{S}_\alpha\|_\infty HL\|u-v\|_h,$$
with $\|\mathcal{S}_\alpha\|_\infty HL<1$. Then, $G$ has a unique fixed point in $l_h^\infty(\mathbb{N}_0;X)$, by the Banach fixed point theorem.
\end{proof}

The next Lemma provide a necessary tool for the use of the Schauder's fixed point theorem, needed in the second main result on existence and uniqueness of solutions to \eqref{eq5.1}.

\begin{lemma}\label{lemma5.2} Let $h$ be an admissible weight  and $U\subset l_h^\infty(\mathbb{N}_0;X)$ such that:
\begin{enumerate}
\item[(a)] The set $H_n(U)=\left\{\frac{u(n)}{h(n)}:u\in U\right\}$ is relatively compact in $X,$ for all $n\in\mathbb{N}_0$.
\item[(b)] $\displaystyle\lim_{n\to\infty}\frac{1}{h(n)}\sup_{u\in U}\|u(n)\|=0$, that is, for each $\varepsilon>0$, there are $N>0$ such that $\displaystyle{\frac{\|u(n)\|}{h(n)}<\varepsilon}$, for each $n\ge N$ and for all $u\in U$.
\end{enumerate}
Then $U$ is relatively compact in $l_h^\infty(\mathbb{N}_0;X)$.
\end{lemma}

\begin{proof} Let $\{u_m\}_m$ be a sequence in $U$, then by (a) for $n\in\mathbb{N}_0$ there is a convergent subsequence $\{u_{m_j}\}_j\subset\{u_m\}_m$ such that $\displaystyle\lim_{j\to\infty}\frac{u_{m_j}(n)}{h(n)}=a(n)$, that is, for each $\varepsilon>0$ there exists $N(n,\varepsilon)>0$ such that $\|\frac{u_{m_j}(n)}{h(n)}-a(n)\|<\varepsilon$ for all $j\ge N(n,\varepsilon)$. Let $\varepsilon>0$ and $N$ the value of the assumption (b). If we consider $N^*:=\displaystyle\min_{0\le n< N}N(n,\varepsilon)$, then for $j,k\ge N^*$ we have
$$ \sup_{0\le n< N}\frac{\|u_{m_j}(n)-u_{m_k}(n)\|}{h(n)} \le \sup_{0\le n< N}\|\frac{u_{m_j}(n)}{h(n)}-a(n)\| + \sup_{0\le n< N}\|\frac{u_{m_k}(n)}{h(n)}-a(n)\|<\varepsilon/2+\varepsilon/2=\varepsilon, $$
and also
$$ \sup_{n\ge N} \frac{\|u_{m_j}(n)-u_{m_k}(n)\|}{h(n)} \le \sup_{n\ge N} \frac{\|u_{m_j}(n)\|}{h(n)} + \sup_{n\ge N} \frac{\|u_{m_k}(n)\|}{h(n)}<\varepsilon/2+\varepsilon/2=\varepsilon. $$
Consequently,
$$ \|u_{m_j}-u_{m_k}\|_h = \sup_{n\in\mathbb{N}_0}\frac{\|u_{m_j}(n)-u_{m_k}(n)\|}{h(n)}<\varepsilon, $$
therefore $\{u_{m_j}\}_j$ is a Cauchy subsequence in $l_h^\infty(\mathbb{N}_0;X)$ which finishes the proof.
\end{proof}

For $f:\mathbb{N}_0 \times X \to X$ we recall that the Nemytskii operator $\mathcal{N}_f:l_h^\infty(\mathbb{N}_0;X) \to l_h^\infty(\mathbb{N}_0;X) $ is defined by
$$
\mathcal{N}_f(u)(n) := f(n,u(n)), \quad n \in \mathbb{N}_0.
$$
The next theorem is the second main result for this section. It gives one  useful criteria for the existence of solutions without use of Lipchitz type conditions.

\begin{thm} \label{th5.3} Let $h$ be an admissible weight function. Let $A$ be the generator of a bounded discrete $\alpha$-resolvent family $\{{\mathcal S}_\alpha(n)\}_{n\in\N_0}$ on a Banach space $X$ for some $1<\alpha\le2$, and $f:\mathbb{N}_0\times X\to X$. Suppose that the following conditions are satisfied:
\begin{enumerate}
\item[(i)] There exist a sequence $M\in l^\infty(\mathbb{N}_0)$ and a function $W:\mathbb{R}^+\to\mathbb{R}^+$, with $W(y)\le Cy$ for $y\in\mathbb{R}^+$, such that $\|f(k,x)\|\le M(k)W(\|x\|)$ for all $k\in\mathbb{N}_0$ and $x\in X$.
\item[(ii)] The Nemytskii operator is continuous in $l_h^\infty(\mathbb{N}_0;X)$, that is, for each $\varepsilon>0$, there is $\delta>0$ such that for all $u,v\in l_h^\infty(\mathbb{N}_0;X)$, $\|u-v\|_h <\delta$ implies that $\|\mathcal{N}_f(u)-\mathcal{N}_f(v)\|_h < \varepsilon$.
\item[(iii)] For all $a\in\mathbb{N}_0$ and $\sigma>0$, the set $\{\mathcal{S}_\alpha(n)f(k,x):0\le k\le a,\|x\|\le\sigma\}$ is relatively compact in $X$ for all $n\in\mathbb{N}_0$.
\end{enumerate}

Then the problem \eqref{eq5.1} has an unique solution in $l_h^\infty(\mathbb{N};X)$.
\end{thm}

\begin{proof} Let us define the operator $G:l_h^\infty(\mathbb{N}_0;X)\to l_h^\infty(\mathbb{N}_0;X)$ given by
$$
Gu(n)=\sum_{k=0}^{n-2}\mathcal{S}_\alpha(n-2-k)f(k,u(k)), \quad n\ge2.
$$
To prove that $G$ has a fixed point in $l_h^\infty(\mathbb{N}_0)$, we will use Leray-Schauder alternative theorem. We verify that the conditions of the theorem are satisfied:\\
\begin{itemize}
\item $G$ is well defined: Let $u\in l_h^\infty(\mathbb{N}_0)$ and $M_\infty:=\displaystyle\sup_{n\in\mathbb{N}_0}M(n)$, then by the assumption (i)
\begin{eqnarray*}
\|Gu(n)\| &\le& \sum_{k=0}^{n-2}\|\mathcal{S}_\alpha(n-2-k)\|\|f(k,u(k))\|
\le \|\mathcal{S}_\alpha\|_\infty\sum_{k=0}^{n-2}M(k)W(\|u(k)\|)\\
&\le& \|\mathcal{S}_\alpha\|_\infty M_\infty C \sum_{k=0}^{n-2}\|u(k)\|
\le \|\mathcal{S}_\alpha\|_\infty M_\infty C\|u\|_h\sum_{k=0}^{n-2}h(k).
\end{eqnarray*}
Therefore, for each $n\in\mathbb{N}_0$, we have
$$
\frac{\|Gu(n)\|}{h(n)} \le \|\mathcal{S}_\alpha\|_\infty M_\infty C\|u\|_h\frac{1}{h(n)}\sum_{k=0}^{n-2}h(k).
$$
Since $h$ is admissible, the claim follows.

\item $G$ is continuous: Let $\varepsilon>0$ and $u,v\in l_h^\infty(\mathbb{N}_0).$  Then, for each $n\in\mathbb{N}_0$,
\begin{eqnarray*}
\|Gu(n)-Gv(n)\|&\le&\sum_{k=0}^{n-2}\|\mathcal{S}_\alpha(n-2-k)\|\|f(k,u(k))-f(k,v(k))\|\\
&\le& \|\mathcal{S}_\alpha\|_\infty\sum_{k=0}^{n-2}\|f(k,u(k))-f(k,v(k))\|\\
&\le& \|\mathcal{S}_\alpha\|_\infty\|\mathcal{N}_f(u)-\mathcal{N}_f(v)\|_h\sum_{k=0}^{n-2}h(k).
\end{eqnarray*}
Therefore
$$ \frac{\|Gu(n)-Gv(n)\|}{h(n)} \le \|\mathcal{S}_\alpha\|_\infty\|\mathcal{N}_f(u)-\mathcal{N}_f(v)\|_h\frac{1}{h(n)}\sum_{k=0}^{n-2}h(k).$$
Hence, by the assumption (ii) and admissibility of $h$ we obtain $\|Gu-Gv\|_h <\varepsilon.$

\item $G$ is compact: For $R>0$ given, let $B_R(l_h^\infty(\mathbb{N}_0;X)):=\{w\in l_h^\infty(\mathbb{N}_0;X):\|w\|_h< R\}$. To prove that $V:=G(B_R(l_h^\infty(\mathbb{N}_0;X)))$ is relatively compact, we will use Lemma \ref{lemma5.2}. We check that the conditions in such Lemma are satisfied:
\begin{itemize}
\item[(a)] Let $u\in B_R(l_h^\infty(\mathbb{N}_0;X))$ and $v=Gu$. We have
$$ v(n)=Gu(n)=\sum_{k=0}^{n-2} \mathcal{S}_\alpha(n-2-k)f(k,u(k))=\sum_{k=0}^{n-2} \mathcal{S}_\alpha(k)f(n-2-k,u(n-2-k)), $$
and then,
$$ \frac{v(n)}{h(n)}=\frac{n-1}{h(n)}\left(\frac{1}{n-1}\sum_{k=0}^{n-2} \mathcal{S}_\alpha(k)f(n-2-k,u(n-2-k))\right). $$
Therefore $\frac{v(n)}{h(n)}\in\frac{n-1}{h(n)}co(K_n)$, where $co(K_n)$ denotes the convex hull of $K_n$ for the set
$$ K_n=\bigcup_{k=0}^{n-2}\{\mathcal{S}_\alpha(k)f(\xi,x):\xi\in\{0,1,2,\dots, n-2\},\|x\|\le R\}, \quad n\in\mathbb{N}_0. $$
Note that each set $K_n$ is relatively compact by the assumption (iii). From the inclusions $H_n(V)=\left\{\frac{v(n)}{h(n)}:v\in V\right\}\subseteq\frac{n-1}{h(n)}co(K_n)\subseteq\frac{n-1}{h(n)}co(\overline{K_n})$, we conclude that the set  $H_n(V)$ is relatively compact in $X,$ for all $n\in\mathbb{N}_0$.\\

\item[(b)] Let $u\in B_R(l_h^\infty(\mathbb{N}_0;X))$ and $v=Gu$. For each $n\in\mathbb{N}_0$, we have
\begin{eqnarray*}
\frac{\|v(n)\|}{h(n)} &\le& \frac{1}{h(n)}\sum_{k=0}^{n-2} \|\mathcal{S}_\alpha(n-2-k)\|\|f(k,u(k))\|\\
&\le& \|\mathcal{S}_\alpha\|_\infty M_\infty C\|u\|_h\frac{1}{h(n)}\sum_{k=0}^{n-2}h(k)
\le \|\mathcal{S}_\alpha\|_\infty M_\infty CR\frac{1}{h(n)}\sum_{k=0}^{n-2}h(k),
\end{eqnarray*}
then the admissibility of $h$ implies $\displaystyle\lim_{n\to\infty}\frac{\|v(n)\|}{h(n)}=0$ independently of $u\in B_{R}(l_h^\infty(\mathbb{N}_0;X))$.
\end{itemize}
Therefore, $V=G(B_R(l_h^\infty(\mathbb{N}_0;X)))$ is relatively compact in $l_h^\infty(\mathbb{N}_0;X)$ by Lemma \ref{lemma5.2} and we conclude that $G$ is a compact operator.

\item The set $U:=\{u\in l_h^\infty(\mathbb{N}_0;X):u=\gamma Gu,\gamma\in(0,1)\}$ is bounded: In fact, let us consider $u\in\l_h^\infty(\mathbb{N}_0)$ such that $u=\gamma Gu$, $\gamma\in(0,1)$. Again by (i),
    $$\|u(n)\|=\|\gamma Gu(n)\|\le\sum_{k=0}^{n-2} \|\mathcal{S}_\alpha(n-2-k)\|\|f(k,u(k))\|\le\|\mathcal{S}_\alpha\|_\infty M_\infty C\|u\|_h\sum_{k=0}^{n-2}h(k).$$
    Then for each $n\in\mathbb{N}_0$ we have
    $$\frac{\|u(n)\|}{h(n)}\le \|\mathcal{S}_\alpha\|_\infty M_\infty C\|u\|_h\frac{1}{h(n)}\sum_{k=0}^{n-2}h(k).$$
  We deduce that $U$ is a bounded set in $l_h^\infty(\mathbb{N}_0;X)$.
\end{itemize}
Finally, by using the Leray-Schauder alternative theorem, we conclude that $G$ has a fixed point $u\in l_h^\infty(\mathbb{N}_0)$.
\end{proof}

\section{The Poisson transformation of fractional difference operators}

For each $n \in \mathbb{N}_0,$ the Poisson distribution is defined by
\begin{equation}\label{poi}
p_n(t):= e^{-t} \frac{t^{n}}{n!}, \quad t \geq 0.
\end{equation}
The Poisson distribution arises in connection with classical Poisson processes and semigroups of functions; note that it is also called fractional integral semigroup in \cite[Theorem 2.6]{Sinclair}.  In this section we study in detail this sequence of functions (Proposition \ref{poisson}), the Poisson transformation (considered deeply in Theorem \ref{th3.2}) and give their connection with fractional difference and differential operators  in Theorem \ref{th3.4}.

\begin{prop}\label{poisson} Let $n \in \mathbb{N}_0$ and $(p_n)_{n\ge 0}$ given by (\ref{poi}). Then
\begin{itemize}

\item[(i)] For $t\ge 0$,  the inequality $p_n(t) \geq 0$ holds,
$
\int_0^{\infty} p_n(t) dt =1,$ and
$$
\int_0^\infty p_n(t)p_m(t)dt= {1\over 2^{n+m+1}}{(n+m)!\over n!m!}, \qquad n,m \in \mathbb{N}_0.
$$

\item[(ii)] The semigroup property $p_n\ast p_m=p_{n+m}$ holds for $n, m\in \mathbb{N}_0$.
\item[(iii)] Given $t\ge 0$, then
$$
(p_{(\cdot)})(t)\ast p_{(\cdot)}(t))(n)=2^ne^{-t}p_n(t), \qquad n\in \mathbb{N}_0.
$$
\item[(iv)] For $m,n \in \mathbb{N}_0$, we have $
\Delta^m p_n=(-1)^mp^{(m)}_{n+m}$.
\item[(v)] The $Z$-transform and the Laplace transform are given by
\begin{eqnarray*}
\widetilde{p_{(\cdot)}(t)}(z)&=&e^{-t(1-{1\over z})}, \qquad z\not=0,\quad t>0;\cr
\widehat{p_n}(\lambda)&=&{1\over (\lambda+1)^{n+1}}, \qquad \Re \lambda >-1,\quad n\in \mathbb{N}_0.
\end{eqnarray*}
\end{itemize}

\end{prop}

\begin{proof} The proof of (i) and (ii) is straightforward, and also may be found in \cite[Theorem 2.6]{Sinclair}. To show (iii), note that
$$
(p_{(\cdot)})(t)\ast p_{(\cdot)}(t))(n)=e^{-2t}{t^n\over n!}\sum_{j=0}^n {n!\over j!(n-j)!}=2^ne^{-t}p_n(t),
$$
for $n\in \mathbb{N}_0$ and $t\ge 0$. Now we get that
$$
(p_{n+1})'(t)= -e^{-t}{t^{n+1}\over (n+1)!}+e^{-t}{t^n\over n!}=-\Delta p_n(t),
$$
and we iterate to obtain the equality $
\Delta^m p_n=(-1)^mp^{(m)}_{n+m}$ for $m,n \in  \mathbb{N}_0$. Finally the $Z$-transform and the Laplace transform of $(p_n)_{n\ge 0}$ are  easily obtained.
\end{proof}

Now  we introduce an integral transform using the Poisson distribution as integral kernel.
Some of their properties are inspired in  results included in \cite[Section 3]{Li14} in particular a remarkable connection between the vector-valued $Z$-transform and the vector-valued Laplace transform, Theorem \ref{th3.2} (ii).

\begin{thm}\label{th3.2} Let $\psi \in L^1(\mathbb{R_+};X)$ and we define  $(\mathcal{P}\psi)\in s(\mathbb{N}_0;X)$ by
\begin{equation}\label{PT}
(\mathcal{P}\psi)(n):= \int_0^{\infty} p_n(t)\psi(t)dt, \quad n \in \mathbb{N}_0.
\end{equation}

Then the following properties hold.

\begin{itemize}
\item[(i)] The map $\mathcal{P}$ defines a bounded linear operator from $L^1(\mathbb{R_+};X)$ to $\ell^1({\mathbb{N}_0};X)$ and $
\| \mathcal{P} \| = 1.
$
\item[(ii)] For $\psi \in L^1(\mathbb{R_+};X)$, we have that
$$
\mathcal{P}(\psi)(n)={(-1)^n\over n!}\left[\widehat{\psi}(\lambda)\right]^{(n)}\vert_{\lambda=1}, \qquad n \in \mathbb{N}_0.
$$
In particular the map $\mathcal{P}$ is inyective.

\item[(iii)] We have that
$\displaystyle{
\widetilde{(\mathcal{P} \psi)}(z) = \widehat \psi (1-1/z),} $ for  $|z| >1.
$
\item[(iv)] For $a \in L^{1}(\mathbb{R}_+)$ and $\psi \in L^{1}(\mathbb{R_+} ;X)$ then
$
\mathcal{P}(a*\psi)=\mathcal{P}(a)*\mathcal{P}(\psi).
$

\item[(v)] If there are constants $M>0$ and $\omega \ge 0$ such that $\|\psi(t) \| \leq Me^{-\omega t}$ for a.e. $t\geq 0$ then $ \|\mathcal{P}(\psi)(n) \| \leq \displaystyle \frac{M}{(1+\omega)^{n+1}}$ for all $n \in \mathbb{N}_0.$ In particular if  $\psi$ is bounded then $\{\mathcal{P}(\psi)(n)$ for $n \in  \mathbb{N}_0$ is well-defined by (\ref{PT}) and $\{\mathcal{P}(\psi)(n)\}_{n\in \mathbb{N}_0}$ is bounded.

\item[(vi)] Let $X$ be a Banach lattice and $\psi(t) \geq 0$ for all $x \geq 0$ and a.e. $t \geq 0$ then $\mathcal{P}(\psi)(n) \geq 0 $ for  $n \in \mathbb{N}_0.$

    \item[(vii)] Suppose that $\{ S(t)\}_{t\geq 0}\subset {\mathcal B}(X)$  is a uniformly bounded family of operators. If $\{ S(t)\}_{t\geq 0}$ is compact then $\{\mathcal{P}(S)(n)\}_{n\in \mathbb{N}_0}$ is compact.\\
\end{itemize}

\end{thm}
\begin{proof}
To prove (i) is enough to observe that
\begin{align*}
\Vert \mathcal{P} \psi\Vert_1\le  \sum_{n=0}^{\infty} \int_0^{\infty} p_n(t) \| \psi(t) \| dt = \int_0^{\infty} \sum_{n=0}^{\infty} \frac{t^n}{n!} e^{-t}  \|\psi(t)\| dt = \int_0^{\infty} \| \psi(t) \|dt=\Vert\psi\Vert_1,
\end{align*}
for $\psi \in L^1(\mathbb{R_+};X)$. Take $0\not= x\in X$ and define $(e_{\lambda}\otimes x)(t):=e^{-\lambda t}x$ for $t,\lambda >0$. Note that $e_\lambda \otimes x\in L^1(\mathbb{R_+};X)$ and $\Vert e_\lambda \otimes x\Vert_1={1\over \lambda}\Vert x\Vert$ for $\lambda >0$. It is straightforward to check that $$\mathcal{P}(e_\lambda \otimes x)(n)={1\over (1+\lambda)^{n+1}}x, \qquad \lambda >0, \quad n\in \mathbb{N}_0,
$$
 and $\Vert \mathcal{P}(e_\lambda \otimes x)\Vert_1={1\over \lambda}\Vert x\Vert$ for $\lambda>0$. We conclude that $\Vert \mathcal{P}\Vert=1$.

By properties of Laplace transform, see for example \cite[Theorem 1.5.1]{ABHN01}, we have that
$$
\mathcal{P}(\psi)(n)={(-1)^n\over n!}\left[\widehat{\psi}(\lambda)\right]^{(n)}\vert_{\lambda=1}, \qquad \psi\in L^1(\mathbb{R_+};X).
$$
Now take $\psi\in L^1(\mathbb{R_+};X)$ such that $\mathcal{P}(\psi)(n)=0$ for all $n\in\mathbb{N}_0$. Then we also get that $\left[\widehat{\psi}(\lambda)\right]^{(n)}\vert_{\lambda=1}=0$ for $n\in\mathbb{N}_0$. Since $\widehat{\psi}$ is an holomorphic function, we conclude that $\widehat{\psi}=0$ and then $\psi=0$ where we apply that the Laplace transform is injective, see for example \cite[Theorem 1.7.3]{ABHN01}.

Part (iii) is proved following similar ideas than in \cite[Theorem 3.1]{Li14}. For (iv) note that because $a \in L^{1}(\mathbb{R}_+)$ and $\psi \in L^{1}(\mathbb{R_+} ;X)$ we have  $a*\psi  \in L^1(\mathbb{R}_+;X)$ and
\begin{eqnarray*}
\mathcal{P}(a*\psi)(n)&=&\int_0^\infty{t^n\over n!}e^{-t}\int_0^ta(s)\psi(t-s)dsdt=\int_0^\infty a(s)e^{-s}\int_0^\infty{(s+u)^n\over n!}e^{-u}\psi(u)duds\\
&=&\sum_{j=0}^\infty\int_0^\infty a(s)e^{-s}{s^j\over j!}ds\int_0^\infty{u^{n-j}\over (n-j)!}e^{-u}\psi(u)du=\left(\mathcal{P}(a)*\mathcal{P}(\psi)\right)(n),
\end{eqnarray*}
for $n\in {\mathbb{N}_0}$.
Assertion (v) and (vi) are easily checked
 and assertion (vii) is obtained from  \cite[Corollary 2.3]{We}.
\end{proof}

We check  Poisson transforms of  some known functions in the next example. Note that, in fact, the Poisson transform can be extended to other sets than $L^1(\mathbb{R_+};X)$, for example, $\mathcal{P}(f)(n)$ is well-defined for measurable functions $f$ such that $$\hbox{ess sup}_{t\ge 0}\Vert e^{\omega t}f(t)\Vert<\infty.$$
Also the identity given in Theorem \ref{th3.2} (iii) holds for the Dirac distribution $\delta_t$ for $t>0$, see Proposition \ref{poisson} (v).

\begin{definition}
The  map $\mathcal{P}: L^1(\mathbb{R_+};X) \to \ell^1({\mathbb{N}_0};X)$ defined by (\ref{PT})
is called  the Poisson transformation.
\end{definition}

\begin{example} \label{exam3.3}(i) Note that $\mathcal{P}(e_{\lambda})(n)=\displaystyle{1\over (\lambda+1)^{n+1}}$ for $n\in \mathbb{N}_0$ , where $e_\lambda(t):=e^{-\lambda t}$ for $t>0$.

(ii) By Proposition \ref{poisson}(i), $$\displaystyle{\mathcal{P}(p_{m})(n)={1\over 2^{n+m+1}}{(n+m)!\over n!m!}}= {1\over 2^{n+m+1}}k^{m+1}(n), \qquad n,m \in \mathbb{N}_0,$$
where the kernel $k^\alpha$ is defined in (\ref{kerne}).

(iii) For $\alpha >0$, define
\begin{equation*}
g_{\alpha}(t) := \left\{%
\begin{array}{cr}
    \displaystyle \frac{t^{\alpha-1}}{\Gamma(\alpha)}, \,   & \quad t>0; \\
     0, \, \,  & \quad  t= 0.\\
\end{array}%
\right.
\end{equation*}
Then  the identity
$
\mathcal{P}(g_{\alpha})= k^{\alpha}$ holds, see more details in \cite[Example 3.3]{Li14}.

(iv) The Mittag-Leffler function is an entire function defined by
$$
E_{\alpha,\beta}(z):=\sum_{k=0}^\infty\frac{z^k}{\Gamma(\alpha k+\beta)}, \quad \alpha,\beta>0, \qquad z\in  \mathbb{C},
$$
see for example \cite[Section 1.3]{Ba01}. Now take $\lambda\in \mathbb{C}$ such that $\vert \lambda\vert <1$ and $s_{\alpha, \beta}(t):= t^{\beta-1}E_{\alpha, \beta}(\lambda t^\alpha)$, for $t>0$. Then
$$
\mathcal{P}(s_{\alpha, \beta})(n)=\sum_{k=0}^\infty\frac{\lambda^k}{n!\Gamma(\alpha k+\beta)}\int_0^\infty t^{n+\alpha k+\beta-1}e^{-t}dt=\sum_{k=0}^\infty\frac{\lambda^k\Gamma(n+\alpha k+\beta)}{n!\Gamma(\alpha k+\beta)}, \qquad n \in \mathbb{N}_0,
$$
which extends the result \cite[Theorem 4.7]{Li14} proved for $\beta=1$. In the particular case $\beta=\alpha$, we get that
$$
\mathcal{P}(s_{\alpha, \alpha})(n)=\sum_{j=0}^\infty\frac{\lambda^j\Gamma(n+\alpha (j+1))}{n!\Gamma(\alpha (j+1))}=\sum_{j=0}^\infty{\lambda^j}k^{\alpha(j+1)}(n), \qquad n \in \mathbb{N}_0.
$$

\end{example}

Now we are interested to establish a notable relation between the discrete and continuous  fractional concepts in the sense of Riemann-Liouville.
In order to give our next result, we recall that the Riemann-Liouville fractional integral of order $\alpha>0$, of
a  locally integrable function $u:\;[0,\infty) \to X$ is given by:
\begin{equation*}
I_t^{\alpha}u(t): = (g_{\alpha}*u)(t) := \int_0^t
g_{\alpha}(t-s)u(s)ds, \qquad t\ge 0.
\end{equation*}
The Riemann-Liouville fractional derivative of order $\alpha,$ for  $m-1< \alpha <m$, $m\in\mathbb{N}$,  is defined by
\begin{equation}\label{difRL}
{D}_t^\alpha u(t) :=  \frac{d^m}{dt^m}\int_0^t
g_{m-\alpha}(t-s)u(s)ds = \frac{d^m}{dt^m}(g_{m-\alpha}*u)(t), \qquad t\ge 0,
\end{equation}
for $u\in C^{(m)}(\mathbb{R}_+;X)$, see for example \cite[Section 1.2]{Ba01}  and \cite[Section 1.3]{Ba-Di-Sc12}. Compare these definitions with  Definitions \ref{def2.3} and \ref{def2.6}.

\begin{thm}\label{th3.4} Let $m\in \mathbb{N}$ and $m-1<\alpha \le m$. Take $u\in C^{(m)}(\mathbb{R}_+;X)$ such that $e_{-\omega}u^{(m)}$ is integrable for some $0<\omega<1$. Then we have
$$
 \mathcal{P}(D_t^\alpha u)(n+m)=\int_0^{\infty} p_{n+m}(t)D^{\alpha}_{t} u(t)dt = \Delta^{\alpha}\mathcal{P}(u)(n), \quad n \in \mathbb{N}_0.
$$
\end{thm}
\begin{proof}

Set $n\in\mathbb{N}_0$,  $m\in \mathbb{N}$ and $u\in C^{(m)}(\mathbb{R}_+;X)$ such that $e_{-\omega}u^{(m)}$ is integrable for some $0<\omega<1$. We integrate by parts $m$-times to get

\begin{eqnarray*}
 \mathcal{P}(D_t^m u)(n+m)&=&\int_0^\infty p_{n+m}(t)D_t^m u(t)dt =\dots= (-1)^m\int_0^\infty p^{(m)}_{n+m}(t)u(t)dt\\&=& \int_0^\infty  \Delta^m p_n(t)u(t)dt=\Delta^{m}\mathcal{P}(u)(n)
\end{eqnarray*}
where we have applied Proposition \ref{poisson} (iv).

  Now consider $m-1<\alpha<m$. By the definition of  Riemann-Liouville fractional derivative, see formula (\ref{difRL}), we have that
\begin{eqnarray*}
 \mathcal{P}(D_t^\alpha u)(n+m)&=&\int_0^\infty p_{n+m}(t)D_t^\alpha u(t)dt = \int_0^\infty p_{n+m}(t)\frac{d^m}{dt^m}(g_{m-\alpha}*u)(t)dt \\&=& \Delta^{m}\mathcal{P}(g_{m-\alpha}*u)(n) = \Delta^{m}\left(k^{m-\alpha}\ast \mathcal{P}(u)\right)(n)\\
 &=&\Delta^m\left(\Delta^{-(m-\alpha)}\mathcal{P}(u)\right)(n) = \Delta^\alpha \mathcal{P}(u)(n),
\end{eqnarray*}
where we have applied Theorem \ref{th3.2} (iv), Example \ref{exam3.3} (iii) and Definition \ref{def2.3}.
\end{proof}

\section{Discrete $\alpha$-resolvent families via Poisson subordination}

We recall the following concept (see \cite{AM}, \cite{KLW13} and references therein).
\begin{definition}\label{def22}
Let $A$ be a closed and linear operator with domain $D(A)$ defined
on a Banach space $X$ and $ \alpha >0.$ We call $A$ the generator
of an $\alpha$-resolvent family if there exists $\omega \geq
0$ and a strongly continuous function $S_{\alpha}:[0,\infty) \to
\mathcal{B}(X)$ (respectively  $S_{\alpha}:(0,\infty) \to
\mathcal{B}(X)$ in case $0<\alpha<1$) such that \, $\{\lambda^{\alpha} \, : \,  \mbox{Re}(\lambda)> \omega \} \subset \rho(A),$ the resolvent set of $A$,  and
\[(\lambda^{\alpha} -A)^{-1}x = \int_0^{\infty}
e^{-\lambda t} S_{\alpha}(t)x dt,  \quad  \mbox{Re}(\lambda) > \omega,
\quad x \in X.\]
In this case, $S_{\alpha}(t)$ is called the $\alpha$-resolvent
family generated by $A$.
\end{definition}

By the uniqueness theorem for the Laplace transform, a
$1$-resolvent family is the same as a $C_0$-semigroup, while a $2$-resolvent family corresponds to a strongly continuous sine
family. See for example \cite{ABHN01} and the references therein for an overview on these concepts.
Some properties of $(S_{\alpha}(t))_{t>0}$ are
included in the following Lemma. For a proof, see for example \cite{KLW13}.

 \begin{lemma}\label{le2.3} Let $\alpha>0$. The following properties hold:

\begin{enumerate}
\item[(i)] $S_{\alpha}(0)=g_{\alpha}(0)I$ (respectively $\displaystyle\lim_{t\to 0^+}\frac{S_{\alpha}(t)x}{g_\alpha(t)}=x$ for all $x\in X$ in case $0<\alpha<1$).

\item[(ii)] $S_{\alpha}(t)D(A) \subset D(A)$ and $AS_{\alpha}(t)x =
S_{\alpha}(t)Ax $ for all $x \in D(A), \, t \geq 0.$

\item[(iii)] For all $x \in D(A):$ $S_{\alpha}(t)x = g_{\alpha}(t) x + \displaystyle
\int_0^t g_{\alpha}(t-s)AS_{\alpha}(s)xds, \, t \geq 0.$

\item[(iv)] For all $x \in X:$ $(g_{\alpha}*S_{\alpha})(t)x \in
D(A)$ and
\[S_{\alpha}(t)x = g_{\alpha}(t) x + \displaystyle A \int_0^tg_{\alpha}(t-s)S_{\alpha}(s)xds, \,\,\, t \geq 0.\]
\end{enumerate}
\end{lemma}
 The next theorem is the main result of this section.

\begin{thm}\label{th5.66} Suppose that $A$ is the generator of an $\alpha$-resolvent family $(S_{\alpha}(t))_{t>0}$  on a Banach space $X,$ of exponential bound less than 1. Then $A$ is the generator of a discrete $\alpha$-resolvent family $(\mathcal{S}_{\alpha}(n))_{ n \in \mathbb{N}_0}$ defined by
\begin{equation*}\label{eq4.2}
\mathcal{S}_\alpha(n):=\mathcal{P}(S_\alpha)(n), \qquad n \in \mathbb{N}_0.
\end{equation*}
\end{thm}
\begin{proof} Take $x\in D(A)$. Since $(A, D(A))$ is a closed operator and the condition in Lemma \ref{le2.3}(ii) we have that
$$
\mathcal{S}_\alpha(n)A(x)= \int_0^\infty p_n(t) S(t)A(x)dt= \int_0^\infty p_n(t) AS(t)(x)dt=A\mathcal{S}_\alpha(n)(x).
$$

From the identity
\begin{equation*}
S_{\alpha}(t)x = g_{\alpha}(t)x + A \int_0^t g_{\alpha}(t-s)S_{\alpha}(s)x ds, 	 \quad t \geq 0,
\end{equation*}
valid for all $x \in X,$ we obtain
\begin{equation*}
\begin{array}{lll}
\mathcal{S}_\alpha(n)x=\mathcal{P}(S_{\alpha})(n)x &=& \mathcal{P}(g_{\alpha})(n) + A\mathcal{P} (g_{\alpha}*S_{\alpha})(n)x =k^\alpha(n)+A\left(k^\alpha \ast \mathcal{S}_\alpha\right)(n)x,
\end{array}
\end{equation*}
where we have applied Example \ref{exam3.3}(iii) and Theorem \ref{th3.2} (iv) and the second condition in Definition \ref{DefResol}. The theorem  is proved.
\end{proof}

\begin{example}\label{example4.3} Consider the  Mittag-Leffler function $E_{\alpha,\beta}$ studied in Example \ref{exam3.3} (iv). Suppose that $A$ is a bounded operator on the Banach space $X$. It then follows from Definition \ref{def22} that
$$
S_\alpha(t)=t^{\alpha-1}E_{\alpha,\alpha}(A t^\alpha), \quad t\ge0,\,\alpha>0
$$
is the $\alpha$-resolvent family generated by $A$. If $\lVert A \rVert <1,$ then $$\mathcal{S}_\alpha(n)x:=\int_0^\infty e^{-t}\frac{t^n}{n!}t^{\alpha-1}E_{\alpha,\alpha}(A t^\alpha)xdt= \sum_{k=0}^\infty \frac{\Gamma(\alpha (k+1)+n)}{\Gamma(\alpha (k+1))\Gamma(n+1)}A^k x, \quad n \in \mathbb{N}_0,$$
for $x\in X$. Compare with Proposition \ref{th4.3}.
\end{example}

\begin{example} Let us recall the definition of $\omega$-sectorial operator. A closed and densely defined operator $A$ is said to be $\omega$-sectorial of angle $\theta$ if there exist $0<\theta<\pi/2$, $M>0$ and $\omega\in\mathbb{R}$ such that its resolvent exists outside the sector $\omega+\Sigma_\theta:=\{\omega+\lambda:\lambda\in\mathbb{C},\abs{\arg(-\lambda)}<\theta\}$ and
$$
\|(\lambda-A)^{-1}\|\le\frac{M}{\abs{\lambda-\omega}}, \quad \lambda\notin\omega+\Sigma_\theta.
$$
Suppose that $A$ is a $\omega$-sectorial operator of angle $\theta < \alpha \pi /2$ and $\omega <0.$ Then $A$ is the generator of a bounded  $\alpha$-resolvent family $(S_{\alpha}(t))_{t>0}$ on $X$ for $1<\alpha < 2$ given by
$$
S_\alpha(t)x=\frac{1}{2\pi i}\int_{\Gamma}e^{\lambda t}(\lambda^\alpha-A)^{-1}xd\lambda,\qquad t>0,\quad x\in X,
$$
where $\Gamma$  is a suitable path where the resolvent operator is well defined. By Theorem \ref{th5.6} and \ref{th3.2}, its Poisson transformation $\mathcal{S}_\alpha(n) :=\mathcal{P}(S_\alpha)(n)$ defines a bounded discrete $\alpha$-resolvent family $\{\mathcal{S}_\alpha(n)\}_{n\in\mathbb{N}_0} \subset \mathcal{B}(X).$


\end{example}

\begin{example}\label{Ex4.7}Suppose that $A$ is the generator of a bounded sine family $(S(t))_{t>0}$ on $X.$ Then $A$ is the generator of a bounded $\alpha$-resolvent family $(S_{\alpha}(t))_{t>0}$ on $X$ for $1<\alpha<2$  given by
$$
S_\alpha(t)x=\int_0^\infty\psi_{\alpha/2,0}(t,s)S(s)xds, \qquad t>0,\quad x\in X,
$$
where $\psi_{\alpha/2,0}(t,s)$ is the stable L\'evy process, see \cite[Corollary 14]{AM}. Then, by \cite[Theorem 3 (vi)]{AM}
$$
\|S_\alpha(t)\|\le M \int_0^\infty  \psi_{\alpha/2,0}(t,s)ds= M g_{\alpha/2}(t),\quad t>0,
$$ and since $S_{\alpha}(0)=0,$ $\frac{1}{2}<\frac{\alpha}{2}<1$ and $(S_{\alpha}(t))_{t>0}$ is strongly continuous we conclude that $(S_{\alpha}(t))_{t>0}$ is bounded. Hence, again by Theorem \ref{th3.2} and \ref{th5.6}, we obtain a bounded discrete $\alpha$-resolvent family $\{\mathcal{S}_\alpha(n)\}_{n\in \mathbb{N}_0} \subset \mathcal{B}(X).$  \end{example}

Our next corollary imposes a natural and useful condition of compactness on a given family of operators in order to obtain existence and uniqueness of solutions.

\begin{thm}\label{cor5.4} Suppose that $A$ is the generator of a bounded sine family $(S(t))_{t>0}$ on $X$ such that $(\lambda-A)^{-1}$ is a compact operator for some $\lambda $ large enough. Let $h$ be an admissible weight  and $f:\mathbb{N}_0 \times X \to X$ satisfying the following conditions:
\begin{enumerate}
\item[(i)] There exist a function $M\in l^\infty(\mathbb{N}_0)$ and a function $W:\mathbb{R}^+\to\mathbb{R}^+$, with $W(x)\le Cx$ for $x\in\mathbb{R}^+$, such that $\|f(k,x)\|\le M(k)W(\|x\|)$ for all $k\in\mathbb{N}_0$ and $x\in X$.
\item[(ii)] The Nemytskii operator $\mathcal{N}_f$ is continuous in $l_h^\infty(\mathbb{N}_0;X).$
\end{enumerate}
Then, for each $1<\alpha \leq 2,$
the problem \eqref{eq5.1} has an unique solution in $l_h^\infty(\mathbb{N}_0;X).$ \end{thm}

\begin{proof} To prove this result we only have to check that the assumption (iii) in Theorem \ref{th5.3} is satisfied. Indeed,
by hypothesis we have that $(\lambda^\alpha-A)^{-1}$ is compact for all $\lambda^{\alpha} \in \rho(A) $ and all $1<\alpha \leq 2.$ By Example \ref{Ex4.7} we obtain that $A$ is the generator of  a bounded $\alpha$-resolvent family  $(S_\alpha(t))_{t>0},$ which is moreover compact by  \cite[Corollary 2.3]{We}. From Theorem \ref{th3.2} (vii) it follows that $\{\mathcal{S}_\alpha(n)\}_{n\in \mathbb{N}_0}$ is compact. Also, for all $a\in\mathbb{N}_0$ and $\sigma>0$, the set $\{f(k,x):0\le k\le a,\|x\|\le\sigma\}$ is bounded because $\|f(k,x)\|\le M(k)C\|x\|\le MC\sigma$ for all $0\le k\le a$ and $\|x\|\le\sigma$. Consequently, the set $\{\mathcal{S}_\alpha(n)f(k,x):0\le k\le a,\|x\|\le\sigma\}$ is relatively compact in $X$ for all $n\in\mathbb{N}_0$.
\end{proof}

\section{Examples, applications and final comments}

In this section, we provide several concrete examples and applications of the abstract results developed in the previous sections. Finally we present some related problems with problem  \eqref{eq1} for $\alpha=2$. 

\begin{example} Let $m:[a,b] \to (0,1)$ be a continuous function. Let $A$ be the multiplication operator given by $Af(x)=m(x)f(x)$ defined on $L^2(a,b).$ We known that $A$ is a bounded operator \cite[Proposition 4.10, Chapter I]{En-Na00}. Since $(\lambda^2-A)^{-1} = \frac{1}{\lambda^2 -m(x)}$ for $\lambda $ sufficiently large, we have by Definition \ref{def22} that $A$ generates a sine family $(S(t))_{t>0}$ on $L^2(a,b)$, given by
$$ S(t)f(x)=\frac{1}{2\sqrt{m(x)}}\left(e^{\sqrt{m(x)}t}-e^{-\sqrt{m(x)}t}\right)f(x), \qquad t>0. $$
Since $0<m(x)<1$ we obtain by subordination
\begin{eqnarray*}
\mathcal{S}(n)f(x) &=& \int_0^\infty p_n(t)S(t)f(x)dt\\
&=& \int_0^\infty e^{-t}\frac{t^n}{n!}\frac{1}{2\sqrt{m(x)}}\left(e^{\sqrt{m(x)}t}-e^{-\sqrt{m(x)}t}\right)f(x)dt\\
&=& \frac{1}{ n! 2\sqrt{m(x)}}\left(\int_0^\infty t^ne^{-(1-\sqrt{m(x)})t}f(x)dt-\int_0^\infty t^ne^{-(1+\sqrt{m(x)})t}f(x)dt\right)\\
&=& \frac{1}{2\sqrt{m(x)}}\left(\frac{1}{(1-\sqrt{m(x)})^{n+1}}-\frac{1}{(1+\sqrt{m(x)})^{n+1}}\right)f(x),
\end{eqnarray*}
for $n\in {n\in \mathbb{N}_0}$. By Theorem \ref{th5.6} and Theorem \ref{th5.66}, we conclude that the fractional difference equation
$$ \Delta^{2} u(n) =  Au(n+2), \quad n \in \mathbb{N}_0, $$
with initial conditions $u(0)=u_0$ and $u(1)=u_1$, admits the explicit  solution
\begin{eqnarray*}
u(n) &=& (\mathcal{S}(n)(I-A)-2\mathcal{S}(n-1))u_0 + \mathcal{S}(n-1)(I-A)u_1\\
&=& \sqrt{A^{-1}} \left(\frac{I-A}{2}(1-\sqrt{A})^{-(n+1)}-\frac{I-A}{2}(1+\sqrt{A})^{-(n+1)}\right)u_0 \\ && - \sqrt{A^{-1}} \left((1-\sqrt{A})^{-n} -(1+\sqrt{A})^{-n}\right)u_0 \\
&&+ \frac{1}{2}\sqrt{A^{-1}} \left((1-\sqrt{A})^{-n}- (1+\sqrt{A})^{-n}\right)(I-A)u_1 ,\quad n\in\mathbb{N}_0.
\end{eqnarray*}
\end{example}

\begin{example} We study the existence of solutions for the problem
\begin{equation}\label{eq5.3}
\left\{%
\begin{array}{rll}
    &\Delta^{\alpha} u(n,x) =  u_{xx}(n+2,x) + \displaystyle\frac{\sin{n}}{1+n^3}\frac{u(n,x)}{1+ \Big(\displaystyle \int_0^{\pi}|u(n,s)|^2ds \Big)^{1/2}},  \quad n \in \mathbb{N}_0, \quad 0<x<\pi;  \\
    &u(0,x) =0 ;\quad u(1,x)  =0;\\ &u(n,0) =0 ;\quad u(n,\pi)  =0;
\end{array}%
\right.
\end{equation}
for $1<\alpha<2$. We will use Corollary \ref{cor5.4}.\\

Let $X=L^2[0,\pi]$ and let us define the operator $A=\displaystyle\frac{\partial^2}{\partial x^2}$,  on the domain
$$D(A)=\{v\in L^2[0,\pi] / v,v' \textrm{ absolutely continuous, } v''\in L^2[0,\pi], v(0)=v(\pi)=0\}.$$
Observe that the operator $A$ can be written as
$$ Av=-\sum_{n=1}^\infty n^2(v,z_n)z_n, \quad v\in D(A), $$
where $z_n(s)=\sqrt{2/\pi}\sin{ns}$, $n=1,2,\dots$, is an orthonormal set of eigenvectors of $A$.\\
Note that $A$ is the infinitesimal generator of a sine family $S(t)$, $t\in\mathbb{R}$, in $L^2[0,\pi]$, given by
$$ S(t)v=\sum_{n=1}^\infty \frac{\sin{nt}}{n}(v,z_n)z_n, \quad v\in L^2[0,\pi]. $$
The resolvent of $A$ is given by
$$ R(\lambda; A)v=\sum_{n=1}^\infty \frac{1}{\lambda+n^2}(v,z_n)z_n, \quad v\in L^2[0,\pi], -\lambda\neq k^2, k\in\mathbb{N}. $$
The compactness of $R(\lambda; A)$ follows from the fact that eigenvalues of $R(\lambda; A)$ are $\lambda_n=\frac{1}{\lambda+n^2}$, $n=1,2,\dots$, and thus $\displaystyle\lim_{n\to\infty}\lambda_n=0$, see for example \cite{TW}.\\

Let us consider the weighted space
$$ l_h^\infty(\mathbb{N}_0; L^2[0,\pi])=\left\{\xi:\mathbb{N}_0\to L^2[0,\pi]/\sup_{n\in\mathbb{N}}\frac{\|\xi(n)\|_{L^2}}{nn!}<\infty \right\}, $$
where the function $h(n)=nn!$ is an admissible weight function (see Example \ref{example5.3}).

For the function $f:\mathbb{N}_0\times L^2[0,\pi]\to L^2[0,\pi]$, defined by $f(n,v):=\displaystyle\frac{\sin{n}}{1+n^3}\frac{v}{1+\|v\|}$, we consider the Nemystkii operator $\mathcal{N}_f$ associated to $f.$ That is, $N_f(u):\mathbb{N}_0\to L^2[0,\pi]$ is such that $N_f(u)(n):=f(n,u(n))$ for $u:\mathbb{N}_0\to L^2[0,\pi].$ Then:
\begin{enumerate}
\item[(i)] There exists $M(n)=\frac{1}{1+n^3}$ in $l^{\infty}(\mathbb{N}_0)$ and $W(t):=  \frac{t}{1+t}$ such that $\|f(n,v)\|\le M(n) W(\|v\|) $, for all $n\in\mathbb{N}_0$ and $v\in L^2[0,\pi]$.\\
\item[(ii)] For $u_1,u_2\in l_h^\infty(\mathbb{N}_0; L^2[0,\pi])$ and each $n\in\mathbb{N}_0$, we have
\begin{eqnarray*}
\|N_f(u_1)(n)-N_f(u_2)(n)\| &\le& \|\frac{u_1(n)}{1+\|u_1(n)\|}-\frac{u_2(n)}{1+\|u_2(n)\|}\|\\
&=& \|\frac{(u_1(n)-u_2(n))(1+\|u_2(n)\|)+u_2(n)(\|u_2(n)\|-\|u_1(n)\|)\|}{(1+\|u_1(n)\|)(1+\|u_2(n)\|)}\|\\
&\le& \frac{\|u_1(n)-u_2(n)\|}{1+\|u_1(n)\|}+\frac{\|u_2(n)\|\|u_1(n)-u_2(n)\|}{(1+\|u_1(n)\|)(1+\|u_2(n)\|)}\\
&\le& \frac{2}{1+\|u_1(n)\|}\|u_1(n)-u_2(n)\| \le 2\|u_1(n)-u_2(n)\|.
\end{eqnarray*}
\end{enumerate}
Consequently, by Corollary \ref{cor5.4}, we conclude that the problem \eqref{eq5.3} has an unique solution $u\in l_h^\infty(\mathbb{N}_0)$, that is, $u$ satisfies
$$\sup_{n\in\mathbb{N}_0} \frac{\|u(n)\|_{L^2}}{nn!}=\sup_{n\in\mathbb{N}_0} \frac{1}{nn!}\left(\int_0^{\pi}|u(n,x)|^2dx\right)^{1/2}<\infty,$$
Therefore, there exist a constant $K>0$ such that
$$\int_0^{\pi}|u(n,x)|^2dx<K(nn!)^2, \quad n\in\mathbb{N}.$$
\end{example}

\subsection*{Final comments}

In some circumstances, the equation \eqref{eq1} for $\alpha=2$ may have  a different format on the right hand side. For instance, the problem
\begin{equation}\label{eq6.2}
\left\{%
\begin{array}{rll}
    \Delta^{2} u(n) &=  Bu(n+1)+g(n,u(n)),  \quad n \in \mathbb{N}_0;  \\
    u(0) &=u_0 ;\\ u(1) & =u_1.
\end{array}%
\right.
\end{equation}
where $B$ is a linear operator defined on a Banach space $X.$ In such cases, and under mild conditions, we can still handle this problem with our theory. That is the content of the following two results.

\begin{prop}\label{propbat} Let $B$ be a linear operator defined on a Banach space $X$, such that $-2\in\rho(B)$. Then, \eqref{eq6.2} is equivalent to the problem
\begin{equation}\label{eq6.3}
\left\{%
\begin{array}{rll}
    \Delta^{2} u(n) &=  Tu(n+2)+Tu(n)+(I-T)g(n,u(n)),  \quad n \in \mathbb{N}_0;  \\
    u(0) &=u_0 ;\\ u(1) & =u_1.
\end{array}%
\right.
\end{equation}
where $T=I-2(2+B)^{-1}.$
\end{prop}

\begin{proof}
From the definition
\begin{eqnarray*}
    \Delta^2u(n)=u(n+2)-2u(n+1)+u(n),
\end{eqnarray*}
we obtain
\begin{equation*}
    u(n+1)=\frac{1}{2}(u(n+2)-\Delta^2u(n)+u(n)).
\end{equation*}
On the other hand, by \eqref{eq6.2} we have
\begin{equation*}
    u(n+2)-2u(n+1)+u(n)=Bu(n+1)+g(n,u(n))
\end{equation*}
that is, for $-2\in\rho(B)$ we have
\begin{equation*}
    u(n+1)=(2+B)^{-1}u(n+2)+(2+B)^{-1}u(n)-(2+B)^{-1}g(n,u(n)).
\end{equation*}
By identifying both expressions for $u(n+1)$, we achieve
\begin{equation*}
    (2+B)^{-1}u(n+2)+(2+B)^{-1}u(n)-(2+B)^{-1}g(n,u(n))=\frac{1}{2}(u(n+2)-\Delta^2u(n)+u(n)),
\end{equation*}
and therefore
\begin{equation*}
    \Delta^2u(n)=(I-2(2+B)^{-1})u(n+2)+(I-2(2+B)^{-1})u(n)+2(2+B)^{-1}g(n,u(n)).
\end{equation*}
So, assuming $-2\in\rho(B)$, the original problem \eqref{eq6.2} is equivalent to the problem \eqref{eq6.3}, with $T=I-2(2+B)^{-1}.$ \end{proof}

Observe that the operator $T$ in the above proposition is bounded whenever $B$ is a closed linear operator and $-2\in \rho(B).$ A second case of interest is the following.

\begin{prop}\label{propn} Let $B$ be a linear operator defined on a Banach space $X$, such that $1\in\rho(B)$. Then, the problem
\begin{equation}\label{eqn}
\left\{%
\begin{array}{rll}
    \Delta^{2} u(n) &=  Bu(n)+g(n+1,u(n+1)),  \quad n \in \mathbb{N}_0;  \\
    u(0) &=u_0 ;\\ u(1) & =u_1.
\end{array}%
\right.
\end{equation}
is equivalent to the problem
\begin{equation}\label{eq6.3bis}
\left\{%
\begin{array}{rll}
    \Delta^{2} u(n) &=  Tu(n+2)-2Tu(n+1)+(I-T)g(n+1,u(n+1)),  \quad n \in \mathbb{N}_0;  \\
    u(0) &=u_0 ;\\ u(1) & =u_1.
\end{array}%
\right.
\end{equation}
where $T=I-(I-B)^{-1}.$
\end{prop}

\begin{proof}
From the definition
\begin{eqnarray*}
    \Delta^2u(n)=u(n+2)-2u(n+1)+u(n),
\end{eqnarray*}
we obtain
\begin{equation*}
    u(n)=\Delta^2u(n)-u(n+2)+2u(n+1).
\end{equation*}
On the other hand, by \eqref{eqn} we have
\begin{equation*}
    u(n+2)-2u(n+1)+u(n)=Bu(n)+g(n+1,u(n+1))
\end{equation*}
that is, for $1\in\rho(B)$ we have
\begin{equation*}
    u(n)=-(I-B)^{-1}u(n+2)+2(I-B)^{-1}u(n+1)+(I-B)^{-1}g(n+1,u(n+1)).
\end{equation*}
By identifying both expressions for $u(n)$, we achieve
\begin{equation*}
    (I-B)^{-1}\left(-u(n+2)+2u(n+1)+g(n+1,u(n+1))\right)=\Delta^2u(n)-u(n+2)+2u(n+1),
\end{equation*}
and therefore
\begin{equation*}
    \Delta^2u(n)=(I-(I-B)^{-1})u(n+2)-2(I-(I-B)^{-1})u(n+1)+(I-B)^{-1}g(n+1,u(n+1)).
\end{equation*}
So, assuming $1\in\rho(B)$, the problem \eqref{eqn} is equivalent to the problem \eqref{eq6.3bis}, with $T=I-(I-B)^{-1}.$
\end{proof}

For instance, let $B$ be a linear operator defined on a Banach space $X$, and $\gamma$ a positive constant. We study the existence of solutions of the problem
\begin{equation}\label{bat}
\left\{%
\begin{array}{rll}
    &&\Delta^{2} u(n,x) =  (B+2\gamma)u(n+1,x),  \quad n \in \mathbb{N}_0, x\in [a,b];  \\
    &&u(0,x) =0 ;\quad u(1,x) =0,\\
    &&u(n,a) =0 ;\quad u(n,b) =0.
\end{array}%
\right.
\end{equation}
By Proposition \ref{propbat} the solution of \eqref{bat} coincides with the solution of the problem \begin{equation}\label{batb}
\left\{%
\begin{array}{rll}
    &&\Delta^{2} u(n,x) =  Tu(n+2,x)+Tu(n,x),  \quad n \in \mathbb{N}_0, x\in [a,b];  \\
    &&u(0,x) =0 ;\quad u(1,x) =0,\\
    &&u(n,a) =0 ;\quad u(n,b) =0.
\end{array}%
\right.
\end{equation}
 where $T=I-2(2(1+\gamma)+B)^{-1},$ provided that $2+2\gamma \in\rho(-B).$

 As an example of application to Theorem \ref{thlip} with $\alpha=2$, let us consider $X=L^2(\pi,2\pi)$ and define  $$
 Bf(x)= 2\Big( \frac{1}{1+x}- (1+\gamma) \Big)f(x), \quad x \in [\pi, 2\pi].
 $$
 Note that $B$ is bounded. A computation shows that $Tf(x)= -xf(x)$ and therefore generates the sine family
 $$
 S(t)f(x)= \frac{\sin(\sqrt{x}t)}{\sqrt{x}}f(x), \quad x \in [\pi, 2\pi]
 $$
It follows that $ \| T \| \leq 2\pi$ and $
\|S\|_\infty \leq \sqrt{\pi}. $

Let $h$ the admissible weight function defined by $h(n)=nn!$, for which we have $$\sup_{n \in \mathbb{N}_0} \frac{1}{h(n)} \sum_{k=0}^{n-2}h(k)=\frac{1}{18}$$ since $\displaystyle\frac{1}{h(n)}\sum_{k=0}^{n-2}h(k)$ is a decreasing sequence for $n\ge3.$ Let us consider the function $f:\mathbb{N}_0\times X\to X$ defined by $f(n,\xi)=T\xi$. Then $f(n,0)=0$ for all $n\in\mathbb{N}_0$ and the function $f$ verifies:
\begin{enumerate}
\item[(L)] There exists $L:= \| T \|$ such that
$$\|f(n,x)-f(n,y)\|\le \|T\|\|x-y\|,$$
\end{enumerate}
for all $x,y \in X.$ Moreover,
$$
\displaystyle\|T\|\|S\|_\infty\frac{1}{18} < \frac{2 \pi \sqrt{\pi}}{18} <1.
$$
Therefore,  by Theorem \ref{thlip} we conclude  that the problem \eqref{batb} has an unique solution  or, equivalently, the problem \eqref{bat} has an unique solution $u\in l_h^\infty(\mathbb{N}_0;X)$.


\begin{thebibliography}{99}

\bibitem{AL} L. Abadias and C. Lizama. {\it Almost automorphic mild solutions to fractional partial difference-differential equations.} Applicable Analysis, to appear.

\bibitem{AM} L. Abadias and P. J. Miana. {\it A subordination principle on Wright functions and regularized resolvent families.} J. Funct. Spaces, Art. ID 158145, (2015), 9 pp.




\bibitem{ACL14} R.P. Agarwal, C. Cuevas and C. Lizama, \emph{Regularity of Difference Equations on Banach Spaces}, Springer-Verlag, 2014.




\bibitem{Ab-At12} T. Abdeljawad and  F.M. Atici. \emph{ On the
definitions of nabla fractional operators.} Abstract and Applied
Analysis, volume 2012, (2012), 1-13.
doi:10.1155/2012/406757.





\bibitem{ABHN01} {W. Arendt, C. Batty, M. Hieber \and F.
Neubrander.} Vector-valued Laplace Transforms and Cauchy Problems.
Monographs in Mathematics.  vol. 96. Birkh\"auser, Basel, 2001.



\bibitem{AE2007} F. M. Atici and P. W. Eloe. \emph{A transform method in discrete fractional calculus.} Int. J.
Difference Equ. 2 (2) (2007),  165--176.



\bibitem{AE2009b} F. M. Atici and P. W. Eloe. \emph{Discrete fractional calculus with the nabla
operator.} Electronic Journal of Qualitative Theory of
Differential Equations, 3 (2009), 1-12.

\bibitem{AS2010} F.M. Atici and S. Seng\"ul. \emph{Modeling with fractional difference equations.} J. Math. Anal.  Appl., 369 (2010), 1-9.

\bibitem{AE2011} F.M.  Atici and P.W.  Eloe. {\it  Linear systems of fractional nabla difference equations.} Rocky Mountain J. Math. 41 (2) (2011), 353--370.



\bibitem{AtEl11} F. M.  Atici and P. Eloe. {\it Two-point boundary value problems for finite fractional difference equations.} J. Differ. Equ. Appl. 17 (2011), 445-456.

\bibitem{Ba01}  E. Bazhlekova.  {\it Fractional Evolution Equations in
Banach Spaces.} Ph.D. Thesis, Eindhoven University of Technology,
2001.

\bibitem{Ba-Di-Sc12} D. Baleanu, K. Diethelm, E. Scalas and J.J. Trujillo, {\it Fractional Calculus: Models and Numerical Methods.} World Scientific, 2012.



\bibitem{Ba43} H. Bateman, {\it  Some simple differential difference equations and the related functions. } Bull. Amer. Math. Soc. 49 (1943), 494--512.










\bibitem{CeKiNe12} J. Cerm\'ak, T. Kisela, and L. Nechv\'atal.\emph{ Stability and asymptotic properties of a linear fractional difference equation.} Adv. Difference Equ. 122 (2012), 1-14.

\bibitem{CeKiNe13} J. Cerm\'ak, T. Kisela, and L. Nechv\'atal. \emph{Stability regions for linear fractional differential systems and their discretizations.} Appl. Math. Comput. 219 (12) (2013), 7012--7022.



\bibitem{Ch-He-Mo} R.E. Chandler, R. Herman, E.W. Montroll, {\it Traffic dynamics: studies in car following.} Operations Research (1958), 165--184.

\bibitem{Cu-Lu-Pa06} E. Cuesta, C.  Lubich and C. Palencia. {\it Convolution quadrature time discretization of fractional diffusion-wave equations.} Math. Comp. 75 (254) (2006), 673--696.


\bibitem{Cu07} E.  Cuesta. {\it Asymptotic behaviour of the solutions of fractional integro-differential equations and some time discretizations.} Discrete Contin. Dyn. Syst. 2007, Dynamical Systems and Differential Equations. Proceedings of the 6th AIMS International Conference, suppl., 277--285.

\bibitem{Da-Ba-Ka14} I.K. Dassios, D.I. Baleanu and G.I. Kalogeropoulos. {\it On non-homogeneous singular systems of fractional nabla difference equations.} Appl. Math. Comput. 227 (2014), 112--131.



\bibitem{En-Na00} K.J. Engel and R. Nagel. \emph{ One-parameter semigroups for linear evolution equations.}  Graduate Texts in Mathematics, 194. Springer-Verlag, New York, 2000.















\bibitem{KLW13} V. Keyantuo, C. Lizama  and M. Warma. \emph{Spectral criteria for solvability of boundary value problems and positivity of solutions of time-fractional differential equations.} Abstract and Applied Analysis, vol. 2013,  (2013) 1-11. doi:10.1155/2013/614328.


\bibitem{Li14} C. Lizama. {\it The Poisson distribution, abstract fractional difference equations, and stability.} Submitted.

\bibitem{Li15} C. Lizama. {\it $l_p$-maximal regularity for fractional difference equations on UMD spaces.} Math. Nach., to appear.


\bibitem{Lu-Sl-Th96} Ch. Lubich, I. H. Sloan, and V. Thom\'ee. {\it  Nonsmooth data error estimates for approximations
of an evolution equation with a positive type memory term.} Math. Comp., 65 (1996), 1--17.



\bibitem{MR1989} K. S. Miller and B. Ross. \emph{Fractional difference calculus.} In: Univalent functions, fractional
calculus, and their applications (K\"oriyama,
1988), Horwood, Chichester,  (1989), 139--152.

\bibitem{Mi93}  M. Misawa.\emph{A Harnack inequality for solutions of difference differential equations of elliptic-parabolic type.} Math. Z. 213 (3) (1993),  393--423.


\bibitem{pbm} A.P. Prudnikov, Yu.A. Brychkov and O.I. Marichev. \emph{Integrals and series. Volume 1.} Gordon and Breach, 1986.

\bibitem{Sanz88} J. M. Sanz-Serna, {\it A numerical method for a partial integro-differential equation.} SIAM J.
Numer. Anal., 25 (1988),  319--327.






\bibitem{Sinclair} A. M. Sinclair: { \it Continuous semigroups in Banach algebras,} London Mathematical Society Lecture Note Series. {\bf 63,} Cambridge University Press (1982).

\bibitem{TW} C.C. Travis and G.F. Webb. \emph{Compactness, regularity, and uniform continuity properties of strongly continuous cosine families.} Houston Journal of Mathematics 3 (4) (1977) 555-567.


\bibitem{We} L.W. Weis. \emph{A generalization of the Vidav-Jorgens perturbation theorem for semigroups and its application to Transport Theory.} J.  Math. Anal. Appl. 129 (1988), 6-23.





\bibitem{Zygmund} A. Zygmund. {\it Trigonometric Series.} 2nd ed. Vols. I, II, Cambridge University
Press, New York, 1959.

\end{thebibliography}
\end{document}